\documentclass[paper=letter, DIV=12, USenglish, numbers=noenddot]{scrartcl} 
\usepackage{cmap}		
\usepackage[utf8]{inputenc}	
\usepackage[T1]{fontenc}	
\usepackage[USenglish]{babel}
\usepackage[]{microtype}
\usepackage{xfrac}

\usepackage{amsmath}
\usepackage{amssymb}
\usepackage{mathtools}		

\usepackage{aliascnt}		
\usepackage{amsthm}

\usepackage{tikz}		
\usetikzlibrary{matrix,arrows,patterns,intersections,calc,decorations.pathmorphing}

\usepackage[pdftitle={Cusp excursion in hyperbolic manifolds},pdfauthor={Anja Randecker, Giulio Tiozzo},pdfborder={0 0 0}]{hyperref}   
\usepackage[figure, table]{hypcap}	

\newtheoremstyle{break}
 {} 
 {} 
 {\itshape} 
 {} 
 {\bfseries} 
 {} 
 {\newline} 
 {\thmname{#1}\thmnumber{ #2}\thmnote{ (#3)}} 
 
\newtheoremstyle{breakdef}
 {} 
 {} 
 {} 
 {} 
 {\bfseries} 
 {} 
 {\newline} 
 {\thmname{#1}\thmnumber{ #2}\thmnote{ (#3)}} 
 
\newtheoremstyle{remark}
 {} 
 {} 
 {} 
 {} 
 {\itshape} 
 {.} 
 {0.5em} 
 {\thmname{#1}\thmnumber{ #2}\thmnote{ {\normalfont (}#3{\normalfont )}}} 

\theoremstyle{breakdef}
\newtheorem{definition}{Definition}[section]

\theoremstyle{remark}

\newaliascnt{rem}{definition}  

\aliascntresetthe{rem}

\newaliascnt{exa}{definition}  

\aliascntresetthe{exa} 

\newaliascnt{lem}{definition}  
\newtheorem{lem}[lem]{Lemma}
\aliascntresetthe{lem}

\newaliascnt{conv}{definition}  

\aliascntresetthe{conv}

\theoremstyle{break}

\newtheorem{thm}{Theorem}

\newaliascnt{prop}{definition}  
\newtheorem{prop}[prop]{Proposition}
\aliascntresetthe{prop}

\newaliascnt{conj}{definition}  

\aliascntresetthe{conj}

\newaliascnt{cor}{definition}  
\newtheorem{cor}[cor]{Corollary}
\aliascntresetthe{cor}

\numberwithin{figure}{section}			

\DeclareMathOperator{\Isom}{Isom}
\DeclareMathOperator{\SL}{SL}

\newcommand{\Mthick}{{M_{\mathrm{thick}}}}
\newcommand{\dthick}{{d_{\mathrm{thick}}}}
\newcommand{\drel}{{d_{\mathrm{rel}}}}
\newcommand{\kth}{$k^\textrm{\scriptsize th}$\ }
\newcommand{\nminusoneth}{$(N-1)^\textrm{\scriptsize th}$\ }

\usepackage{enumerate}

\newcommand\blfootnote[1]{
  \begingroup
  \renewcommand\thefootnote{}\footnote{#1}
  \addtocounter{footnote}{-1}
  \endgroup
}

\author{Anja Randecker\thanks{University of Toronto, 40 St George St, Toronto ON, Canada, \texttt{anja@math.toronto.edu}.}, Giulio Tiozzo\thanks{University of Toronto, 40 St George St, Toronto ON, Canada, \texttt{tiozzo@math.toronto.edu}. \newline Partially supported by NSERC and the Alfred P. Sloan Foundation.}}

\title{Cusp excursion in hyperbolic manifolds \\ and singularity of harmonic measure}
\date{\today}

\begin{document}
\renewcommand{\sectionautorefname}{Section}
\renewcommand{\subsectionautorefname}{Subsection}
 
\maketitle

\begin{abstract}
 We generalize the notion of cusp excursion of geodesic rays by introducing for any~$k\geq 1$ the \emph{\kth excursion} in the cusps of a hyperbolic $N$--manifold of finite volume.
 We show that on one hand, this excursion is at most linear for geodesics that are generic with respect to the hitting measure of a random walk. On the other hand, for~$k = N-1$, the \kth excursion is superlinear for geodesics that are generic with respect to the Lebesgue measure.
 We use this to show that the hitting measure and the Lebesgue measure on the boundary of hyperbolic space~$\mathbb{H}^N$ for any $N\geq 2$ are mutually singular.
\end{abstract}

\blfootnote{\textup{2010} \textit{Mathematics Subject Classification}: \textup{60G50, 53D25, 60G30}}
\blfootnote{key words: random walks, cusp excursion, harmonic measure, hitting measure}

\enlargethispage*{1cm}
\vspace*{-0.4cm}

\section{Introduction} 

In the Poincar\'e disk, it is well-known that a Brownian motion starting at the center converges almost surely to the boundary of the disk, 
and the resulting hitting measure on $\partial \mathbb{D}$ coincides with Lebesgue measure. 
Other choices of base point lead also to measures in the Lebesgue measure~class. 

A discrete version of the Brownian motion is provided by a random walk.   
Let $X = \mathbb{H}^N$ and let $\Gamma$ be a discrete group of isometries of $X$. Then for any measure $\mu$ on 
$\Gamma$, we can define a random walk 
$$w_n \coloneqq g_1 \dots g_n$$
where $(g_i)$ are i.i.d.~elements of $\Gamma$, with distribution $\mu$.
If $\Gamma$ is non-elementary, then the random walk converges almost surely to the boundary of $X$, and one has the \emph{hitting measure} on $\partial X \cong S^{N-1}$ defined as
$$\nu(A) \coloneqq \mathbb{P}\left(\lim_{n \to \infty} w_n x \in A\right)$$
which is also the unique \emph{$\mu$--harmonic measure} on $\partial X$. This elicits the

\smallskip

\textbf{Question.} Is the hitting measure $\nu$ absolutely continuous with respect to Lebesgue measure?

\medskip

This question has a long history (see \cite{KLP}), starting with Furstenberg (\cite{furstenberg}, \cite{fur71}), who proved that for any lattice~$\Gamma$ in a semisimple Lie group, there exists a random walk on $\Gamma$ which produces an absolutely continuous measure on the boundary. This is the starting point for Furstenberg's rigidity theory.
This approach was generalized by Lyons--Sullivan \cite{Lyons-Sullivan}. Moreover, Connell--Muchnik \cite{CM-Conformal} showed that a random walk with absolutely continuous hitting measure exists on any hyperbolic group which acts cocompactly on a manifold.

All measures constructed with these methods have infinite support and finite first moment in the hyperbolic metric.
The situation changes if one considers measures with stronger moment conditions, for instance with finite support.
Indeed, if $\Gamma = \SL(2, \mathbb{Z}) < \Isom (\mathbb{H}^2)$, then Guivarc'h--LeJan \cite{Guivarch-winding} proved that a measure with finite support produces a singular measure on the boundary. This fact was proven by different techniques also in \cite{BHM} and \cite{DKN}. 
The salient feature of all these examples is that the action of $\Gamma$ on the hyperbolic plane is not cocompact and the quotient manifold has a cusp. 

In \cite{gadre_maher_tiozzo_15}, singularity of the harmonic measure for non-uniform lattices in $\Isom(\mathbb{H}^2)$ is proven using the notion of \emph{excursion} in the cusp. This approach is also applied to Teichm\"uller space in~\cite{gadre_maher_tiozzo_17}.
The main idea is that \emph{generic geodesics with respect to the harmonic measure wander less deeply into the cusp than generic geodesics with respect to the Lebesgue measure}. 

To make this idea precise, \cite{gadre_maher_tiozzo_15} compare three metrics on $\Gamma$: the \emph{word metric} $\Vert \cdot \Vert$ given by fixing a generating set, the \emph{hyperbolic metric} $d$ inherited by the constant curvature metric on $\mathbb{H}^N$, and the \emph{relative metric} $\drel$, which is obtained by collapsing the horoballs to sets of finite diameter. 
In particular, the \emph{word length ratio} is there defined as the ratio between word length and time (which is the same as hyperbolic length) along generic geodesics. 

If $N \geq 3$, though, the parabolic subgroup that stabilizes the cusp has rank greater than $1$, hence the word length ratio is almost surely finite for both the harmonic and Lebesgue measure (cf.~the remark after the proof of \autoref{cor:excursion_Lebesgue_case_superlinear}).
In this paper, we introduce a related notion of excursion to apply the described method to higher-dimensional hyperbolic manifolds with cusps. 

Let $\Gamma < \Isom(\mathbb{H}^N)$ be a discrete group of isometries of hyperbolic $N$--space such that the quotient  
$M = \mathbb{H}^N/\Gamma$ has finite volume and is not compact. Note that the quotient $M$ does not have to be a manifold, it can be an orbifold.
We consider a $\Gamma$--invariant, disjoint collection of horoballs~$\mathcal{H}$ in $\mathbb{H}^N$ which projects to the cusps of $M$.
In order to quantify how deeply geodesics wander into the cusp, let us define our notion of \emph{excursion}.
Given a geodesic ray $\gamma$ and a horoball $H$, that intersects $\gamma$ in $\gamma(t_1)$ and $\gamma(t_2)$, we define the excursion of $\gamma$ in $H$ as $E(\gamma, H) = d_{\partial H}(\gamma(t_1), \gamma(t_2))$, where $d_{\partial H}$ is the distance along the boundary of the horoball (see \autoref{fig:excursion_single_horoball}).

We are interested in the sum of all the excursions along the geodesic up to a given time.
Thus, for any time $t > 0$, let us denote as $\mathcal{H}_t$ the set of horoballs which intersect the geodesic segment 
$[\gamma(0), \gamma(t)]$.
Then for any $k\geq 1$, we define the \emph{\kth excursion} as
$$\mathcal{E}^{(k)}(\gamma, t) \coloneqq \sum_{ H \in \mathcal{H}_t } \ E(\gamma, H)^k$$
and the \emph{average \kth excursion} as the limit
$$\rho^{(k)}(\gamma) \coloneqq \lim_{t \to \infty} \frac{\mathcal{E}^{(k)}(\gamma, t)}{t}.$$

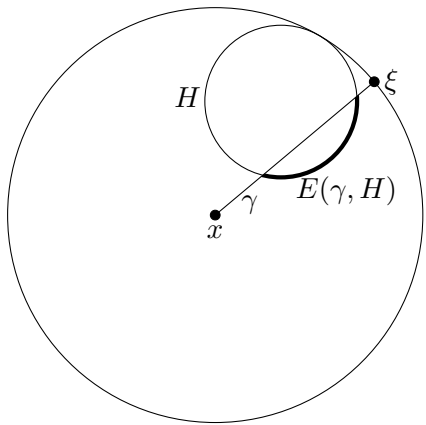
\begin{figure}
 \centering
 \begin{tikzpicture}[scale=0.92]
  \draw[fill] (0,0) node[below]{$x$} circle (2pt);
  \draw (0,0) circle (3cm);
  \draw (0,0) -- (40:3cm);
  \draw[fill] (40:3cm) node[right]{$\xi$} circle (2pt);
  \draw (60:1.9cm) circle (1.1cm);
  \path (-0.4,1.7) node{$H$};
  \path (0.5,0.15) node{$\gamma$};
  \begin{scope}
   \path[clip] (0,0) -- (40:3cm) -- (3,0) -- (0,0);
   \draw[ultra thick] (60:1.9cm) circle (1.1cm);
  \end{scope}
  \path (1.9,0.35) node{$E(\gamma,H)$};
 \end{tikzpicture}
 \label{fig:excursion_single_horoball}
 \caption{The excursion $E(\gamma, H)$ of the geodesic segment $\gamma$ in the horoball $H$ is the length of the thickly drawn arc of the horoball.}
\end{figure}

In our first theorem, we determine the average \kth excursion for generic geodesics.
The choice of a base point $x\in \mathbb{H}^N$ induces a bijection between $\partial \mathbb{H}^N$ and the set of geodesic rays starting at $x$.
Thus, a measure on $\partial \mathbb{H}^N$ can be seen as a measure on the set of geodesic rays and we can speak of generic geodesics with respect to a measure on $\partial \mathbb{H}^N$.

It turns out that average \kth excursions  for geodesics which are generic for the hitting measure 
are different from average \kth excursions for geodesics which are generic for the Lebesgue measure.

\begin{thm}[Cusp excursion] \label{thm:cusp_excursion}
 Let $N \geq 2$ and $x\in \mathbb{H}^N$ be a base point. Let $\Gamma$ be a discrete subgroup of~$\Isom(\mathbb{H}^N)$ such that $\mathbb{H}^N/\Gamma$ is a non-compact hyperbolic orbifold of finite volume.
 
 \begin{enumerate}
  \item Let $k, k' \in \mathbb{R}$ with $k' > k > 1$ and let $\mu$ be a generating measure on $\Gamma$ with finite $(k')^{th}$ moment in some word metric
  and finite exponential moment in the relative metric. Let $\nu$ be the hitting measure for the random walk driven by $\mu$.
  Then for $\nu$--almost every $\xi \in \partial \mathbb{H}^N$ and~$\gamma$ the geodesic ray from $x$ to $\xi$, the average \kth excursion $\rho^{(k)}(\gamma)$ exists and is finite.
  \item Let $\lambda$ be the Lebesgue measure on $\partial \mathbb{H}^N = S^{N-1}$. Then for $\lambda$--almost every $\xi \in \partial \mathbb{H}^N$ and $\gamma$ the geodesic ray from $x$ to $\xi$, the average \nminusoneth excursion $\rho^{(N-1)}(\gamma)$ is infinite.
 \end{enumerate}
\end{thm}

The first part of this theorem will be established in \autoref{sec:excursion_random_walks} and the second part in \autoref{sec:excursion_Lebesgue_case}.
With this theorem, we can directly deduce that the two measures on the boundary of $\mathbb{H}^N$ are mutually singular.

\begin{thm}[Hitting measure and Lebesgue measure are singular] \label{thm:measures_singular}
 Let $N \geq 3$ and $\Gamma$ be a discrete subgroup of~$\Isom(\mathbb{H}^N)$ such that $\mathbb{H}^N/\Gamma$ is a non-compact hyperbolic orbifold of finite volume.
 Let $k \in \mathbb{R}$ with $k > N-1$ and let $\mu$ be a generating measure on $\Gamma$ with finite \kth moment in some word metric, and 
 finite exponential moment in the relative metric.
 Then the hitting measure and the Lebesgue measure on $\partial \mathbb{H}^N$ are mutually singular.
 
\begin{proof}
 By \autoref{thm:cusp_excursion}, the set of boundary points of $\mathbb{H}^N$ such that the average \nminusoneth excursion is finite has zero Lebesgue measure and full hitting measure.
\end{proof}
\end{thm}

This theorem generalizes the case $N=2$ from \cite{gadre_maher_tiozzo_15}. Our $\rho^{(1)}$ plays the role of the word length ratio; however, in the higher dimensional case, the word length ratio has finite first moment with respect to the Liouville measure and cannot be used to establish the singularity of measures. For this reason, we introduce the \kth excursion $\rho^{(k)}$ for $k > 1$. 

From a geometric point of view, the $(N-1)$ in the statement of \autoref{thm:measures_singular} is exactly the Hausdorff dimension of the limit set, which in this case, since $\Gamma$ has finite covolume, equals $\partial \mathbb{H}^N = S^{N-1}$. In fact, we conjecture a similar statement should hold more generally for geometrically finite manifolds by replacing $(N-1)$ with the dimension $\delta$ of the limit set $\Lambda(\Gamma)$. 

Let us also note that for $k > 1$ the excursion $\mathcal{E}^{(k)}$ is not quasi isometric to a metric and it does not satisfy the triangle inequality.
There is also a technical difference between our definition and the definition from \cite{gadre_maher_tiozzo_15}, even for the case $k=1$, since there, the pieces of the geodesic inside the thick part are also added to the excursion.

Statistical properties of the excursion in the cusp have been studied since Sullivan \cite{Sullivan}; let us note, however, that our definition is different. 
In fact, while Sullivan considers the largest distance from the thick part, we take an average over all horoballs traversed by the geodesic. 
The averaging procedure makes this statistical quantity more robust and easier to compare with the random walk~excursion.

Another approach to the singularity of measures is based on the Green metric and the fundamental inequality, as in \cite{BHM} for $N = 2$. 
One of the issues in generalizing this proof is that the group $\Gamma$ is not necessarily hyperbolic for $N \geq 3$, in contrast to the case $N = 2$, so one cannot apply the classical Ancona inequalities. 
However, under the stronger assumption that $\mu$ has finite superexponential moment in the word metric, a generalization of the Ancona inequalities to relatively hyperbolic groups has been proven in \cite{GGPY}, and singularity of harmonic measure has been established in 
 \cite{GGPY} with respect to the Patterson-Sullivan measure and in \cite{GT} with respect to any Gibbs measure (including the Lebesgue measure in variable negative curvature).

We remark that the exponential moment condition on the relative metric in Theorems~\ref{thm:cusp_excursion} and \ref{thm:measures_singular} 
is only used to prove \autoref{L:tau-decay}. It ensures that we can apply the exponential decay estimates from \cite{Sunderland} and \cite{BMSS}.
After this work was completed, we learned from S.~Gou\"ezel that exponential decay estimates hold without any moment condition \cite{Gouezel-decay}; this suggests that the exponential moment condition may be removed from Theorems~\ref{thm:cusp_excursion} and \ref{thm:measures_singular}.

\section{Random walks and hyperbolic geometry}

In this section, we provide some background material on hyperbolic geometry and on random walks. 
In particular, we prove some of the ingredients for the proofs in \autoref{sec:excursion_random_walks}. 

\subsection{Excursion in horoballs}

We first study the excursion of a geodesic segment in a single given horoball. To define the excursion of a geodesic ray or bi-infinite geodesic in all horoballs in \autoref{sec:excursion_random_walks}, we will take into account the excursions in all horoballs that intersect the geodesic.

\pagebreak

\begin{definition}[Excursion in a single horoball]
 Let $H$ be a horoball in $\mathbb{H}^N$ and $\gamma$ be a geodesic segment whose start and end points are not contained in $H$.
 If $\gamma$ and $H$ are not disjoint, then there exists an entry time $t_1$ and an exit time $t_2$ with $t_1 \leq t_2$ such that $\gamma \cap H = \gamma([t_1,t_2])$. Then we define the \emph{excursion of $\gamma$ in $H$} as $E(\gamma, H) = d_{\partial H} (\gamma(t_1), \gamma(t_2))$ where $d_{\partial H}$ is the path metric on $\partial H$ induced by the hyperbolic metric.
  If $\gamma$ and $H$ are disjoint, we set the excursion $E(\gamma, H)$ to be $0$.
\end{definition}

We will now use the geometry of $\mathbb{H}^N$ to prove some lemmas that are used in \autoref{sec:excursion_random_walks}.
In particular, we compute bounds on the excursion in the horoball by considering the intersection with a hyperbolic plane.
Note that the intersection of a horoball with a plane is a Euclidean circle which does not necessarily contain the boundary point of the horoball.

In our proofs it is crucial that the hyperbolic $N$--space is hyperbolic, that is, that there exists a $\delta > 0$ such that for every triangle 
with sides $a, b, c$, the side $c$ is contained in the $\delta$--neighborhood of $a \cup b$.
This $\delta$ is now fixed for the remainder of this article.

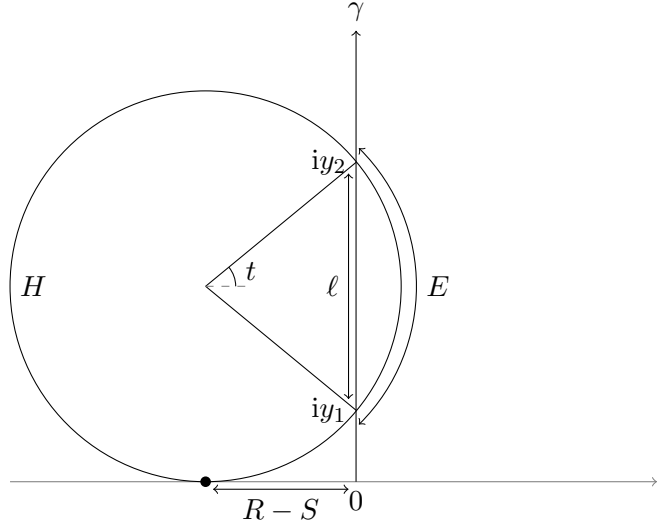
\begin{figure}
 \begin{center}
  \begin{tikzpicture}[scale=2]
   \draw[->, gray] (-2.3,0) -- (2,0) node[right] {};
   \draw[->] (0,0) node[below] {$0$} -- (0,3) node[above] {$\gamma$};
   
   \fill (-1,0) circle (1pt);
   \draw (-1,1.3) circle (1.3);
   \draw (-2.3,1.3) node[right] {$H$};
   \draw[<->] (-0.95,-0.05) -- (-0.05,-0.05);
   \draw (-0.5,-0.05) node[below]{$R-S$};
   
   \draw[<->] (-0.9,1.3) ++(45:1.3) arc (45:-45:1.3);
   \draw (0.4,1.3) node[right] {$E$};
   
   \draw (-1,1.3) -- ++(39.5:1.3);
   \draw (-1,1.3) -- ++(-39.5:1.3);
   \draw (-1,1.3) ++(39.5:0.2) arc (39.5:0:0.2);
   \draw[gray, dashed] (-1,1.3) -- (-0.7,1.3);
   \draw (-0.8,1.4) node[right] {$t$};
   
   \draw (-1,1.3) ++(-39.5:1.3) node[left] {$\mathrm{i} y_1$};
   \draw (-1,1.3) ++(39.5:1.3) node[left] {$\mathrm{i} y_2$};
   
   \draw[<->] (-0.05,0.55) -- (-0.05,2.05);
   \draw (-0.05,1.3) node[left] {$\ell$};
  \end{tikzpicture}
  \caption{In the case $N=2$, the excursion can be calculated as in \autoref{L:comp}.}
  \label{fig:excursion_plane}
 \end{center}
\end{figure}

Our first lemma gives an explicit computation for the excursion in one horoball in the case~$N= 2$. We also get bounds on the excursion when considering another geodesic that does or does not intersect the horoball.

\begin{lem} \label{L:comp}
 Consider the hyperbolic plane $\mathbb{H}^2 = \{ x + \mathrm{i} y, y > 0\}$. Let $\gamma \coloneqq \{x = 0\}$ be a geodesic and let $H$ be a horoball with boundary point to the left of $\gamma$ and which intersects $\gamma$ in $\mathrm{i} y_1$ and $\mathrm{i} y_2$ with $1 \leq |y_1| < |y_2|$ (compare to \autoref{fig:excursion_plane}). 
 
 Then if $\ell$ is the hyperbolic distance between $\mathrm{i} y_1$ and $\mathrm{i} y_2$, the excursion $E(\gamma,H)$ of $\gamma$ in $H$ 
 can be calculated as
 \begin{equation}  
  E(\gamma, H) = 2 \sinh \left( \frac{\ell}{2} \right)
  .
 \end{equation}
 
  Moreover, consider the geodesic $\gamma' \coloneqq \{ x = \delta \}$. 
 Then we have 
 \begin{equation} \label{E:compute}
  \ell \leq 2 \log \left( 1 + \sqrt{\frac{2 \delta}{y_1}} \right) 
 \end{equation}
 if and only if $H$ does not intersect $\gamma'$.
 \end{lem}
 
\begin{proof}
 Let $H$ be a horoball as in \autoref{fig:excursion_plane} where $\mathrm{i} y_1$ and $\mathrm{i} y_2$ are the entry and the exit point of the geodesic $\gamma \coloneqq \{x = 0\}$. 
 Furthermore, let $t$ be half the angle between $\mathrm{i} y_1$ and $\mathrm{i} y_2$, let $R$ be the Euclidean radius of the horoball $H$, and let $S>0$ be such that $(-(R-S),0)$ is the boundary point of the horoball.
 Then we can parameterize the arc of the horoball as
 $$\left\{\begin{array}{l} x = (S-R) + R \cos \theta \\ y = R + R \sin \theta \end{array} \right.$$
 for $-t \leq \theta \leq t$.
 By definition, the excursion is the hyperbolic length of the arc of the horoball.
 We can calculate this length directly with the change of variable of $u = \tan(\theta/2)$ and $\sin \theta = \frac{2u}{1+u^2}$ and~$d\theta = \frac{2}{1+u^2}\, du$:
 $$ E(\gamma, H) = \int_{-t}^t \frac{ds}{y} = \int_{-t}^t \frac{d \theta}{1 + \sin \theta} = 2 \tan t.$$
 Now, the hyperbolic distance between the entry and exit point is 
 $$\ell = \int_{y_1}^{y_2} \frac{dy}{y} = \log \left( \frac{y_2}{y_1} \right) = \log \left( \frac{ 1 + \sin t }{1 - \sin t} \right),$$
 hence
 $$\sin t = \frac{e^\ell  - 1}{e^\ell + 1} \qquad \text{and} \qquad \cos t = \frac{2 e^{\sfrac{\ell}{2}}}{e^\ell + 1},$$
 so 
 $$ E(\gamma,H) = 2 \tan t =  
 e^{\sfrac{\ell}{2}} - e^{- \sfrac{\ell}{2}} = 2 \sinh \left( \frac{\ell}{2} \right).$$
 
 Now let $\gamma' \coloneqq \{ x = \delta \}$ be another geodesic.
 We show that for $S = \delta$, we have $\ell = 2 \log \left( 1 + \sqrt{\frac{2 \delta}{y_1}} \right)$.
 In the situation $S = \delta$, the horoball is tangent to $\gamma'$ and from
 $$y_1 = R(1- \sin t), \qquad \delta = R(1 - \cos t)$$
 it follows
 $$\frac{y_1}{\delta} = \frac{1 - \sin t}{1 - \cos t} = \frac{2}{\left( e^{\sfrac{\ell}{2}} - 1 \right)^2}$$ 
 which, by inverting the formula, yields the claimed equality.
 We obtain the equivalence to the inequality~\eqref{E:compute} then from the fact that for fixed $y_1$, the hyperbolic distance $\ell$ is smaller for~$S < \delta$ than in the case of tangency.
\end{proof}

The following lemma is a generalization of \autoref{L:comp} to higher dimensions. Note that the two geodesics $\gamma$ and $\gamma'$ define a totally geodesic copy of $\mathbb{H}^2$ inside the hyperbolic $N$--space~$\mathbb{H}^N$.
We identify this copy of $\mathbb{H}^2$ with the upper half-plane, and denote its coordinates as $z = x + \mathrm{i} y$, with~$y > 0$.
However, the boundary point of the horoball $H$ does not have to be contained in this plane and hence the intersection of $H$ with $\mathbb{H}^2$ is not necessarily a horoball but can be a disk. By using a suitable rotation of the horoball, we can show that the \kth power of the excursion of $\gamma$ in~$H$ is still bounded as written below.

\begin{lem} \label{lem:excursion_comparison_non-intersecting_geodesics}
 Consider the hyperbolic plane $\mathbb{H}^2$ inside the hyperbolic $N$--space $\mathbb{H}^N$ with coordinates~$x, y$.
 Let $H$ be a horoball 
 which intersects $\gamma \coloneqq \{x = 0\}$ in $\mathrm{i}y_1$ and $\mathrm{i}y_2$ with $1 \leq y_1 < y_2$ but does not intersect $\gamma' \coloneqq \{ x= \delta\}$.

 Then for any $k \in \mathbb{R}$, $k > 1$
 there exists a constant $c> 0$ which depends only on $\delta$ and $k$ such that the excursion satisfies
 \begin{equation*}
  E(\gamma, H)^k \leq c \int_{y_1}^{y_2} \frac{dy}{y^{\frac{k+1}{2}}}
  .
 \end{equation*}
\end{lem}

\begin{proof}
 First note that the horoball $H$ does not necessarily have its boundary point in the same plane as the geodesics $\gamma$ and $\gamma'$. 
 However, one can rotate the horoball $H$ around $\gamma$ to get another horoball~$H'$ which still intersects $\gamma$ in $p_1 \coloneqq \mathrm{i}y_1$ and $p_2 \coloneqq \mathrm{i} y_2$ and moreover its boundary point lies in the boundary of the plane $\mathbb{H}^2$ (and to the left of $\gamma$). 
 Note that the excursions of $\gamma$ in $H$ and in~$H'$ are the same.
 As the boundary point of $H'$ has larger or equal distance to $\gamma'$ than the boundary point of $H$, the horoball $H'$ does not intersect $\gamma'$.
 
 Let $\ell$ be again the hyperbolic distance $d(p_1, p_2) = \log \left(\frac{y_2}{y_1} \right) $.
 We now give an upper bound for $E(\gamma, H)^k$ and a lower bound for $\int_{y_1}^{y_2} \frac{dy}{y^{\frac{k+1}{2}}}$, both in terms of $\ell$.
 In both cases, we use the Taylor expansion for the terms in $\ell$. On one hand, we have from \autoref{L:comp}
 \begin{equation*}
  E(\gamma, H) = 2 \sinh \left( \frac{\ell}{2} \right) = \ell + O(\ell^3)
 \end{equation*}
 for $\ell \to 0$.
 For the other bound, we start with the calculation
 \begin{equation*}
  \int_{y_1}^{y_2} \frac{dy}{y^{\frac{k+1}{2}}} = \frac{2}{k-1} \left( y_1^{-\frac{k-1}{2}} - y_2^{-\frac{k-1}{2}} \right)
  = \frac{2}{k-1} y_1^{-\frac{k-1}{2}} \left( 1 - e^{- \frac{\ell(k-1)}{2} } \right)
  .
 \end{equation*}
 Moreover, since the horoball $H'$ does not intersect $\gamma'$, we can use \autoref{L:comp} and $\log(1+x) \leq x$ to deduce
 \begin{equation} \label{E:l-bound}
  \ell \leq 2 \log \left( 1 + \sqrt{ \frac{2\delta}{y_1} } \right) \leq  \frac{2 \sqrt{2 \delta}}{\sqrt{y_1}}
  .
 \end{equation}
 This implies
 \begin{equation*}
  y_1 \leq \frac{8 \delta}{\ell^2}
  .
 \end{equation*}
 With this bound on $y_1$ and the Taylor expansion $1-e^{-x} = x + O(x^2)$, we have then
 \begin{equation*}
  \int_{y_1}^{y_2} \frac{dy}{y^{\frac{k+1}{2}}} \geq \frac{2}{k-1} \frac{\ell^{k-1}}{(8 \delta)^{\frac{k-1}{2}}}  \left( 1 - e^{- \frac{\ell(k-1)}{2} } \right) 
  = \frac{\ell^{k}}{(8 \delta)^{\frac{k-1}{2}}} + O(\ell^{k+1})
 \end{equation*}
 for $\ell \to 0$.
 
 \pagebreak
 
 Using both of the bounds we obtained, we get
 \begin{equation*}
  \frac{E(\gamma, H)^k}{\int_{y_1}^{y_2} y^{-\frac{k+1}{2}} \ dy } \leq (8 \delta)^{\frac{k-1}{2}} \frac{\ell^k + O(\ell^{k+2})}{\ell^k + O(\ell^{k+1})}
 \end{equation*}
 which is bounded as $\ell \to 0$. Note that $\ell$ is bounded above by equation~\eqref{E:l-bound}, as $y_1 \geq 1$.
 The bound given above is independent of $y_1$ and $y_2$ which proves the existence of the constant $c > 0$ as in the claim.
\end{proof}

Now we consider the case of two parallel geodesics $\gamma$ and $\gamma'$ in $\mathbb{H}^N$, that both intersect a horoball. Note that this horoball again need not have its boundary point in the plane containing $\gamma$ and $\gamma'$.

\begin{lem} \label{L:dist1}
 Consider the hyperbolic plane $\mathbb{H}^2$ inside the hyperbolic $N$--space $\mathbb{H}^N$ with coordinates~$x, y$.
 Let $H$ be a horoball which intersects $\gamma \coloneqq \{x = 0\}$ in $p_1 = \mathrm{i}y_1$ and $p_2 = \mathrm{i}y_2$ with $1 + \delta \leq y_1 < y_2$ and intersects $\gamma' \coloneqq \{ x= \delta\}$ in $p_1' = \delta + \mathrm{i} y_1'$ and $p_2' = \delta + \mathrm{i} y_2'$ with $1 + \delta \leq y_1' < y_2'$ (compare \autoref{fig:intersection_gamma_prime}).
 
 Furthermore, let $d = d(\mathrm{i}, p_1)$ be the distance between the point $\mathrm{i}$ and the entry point of $\gamma$ in $H$. 
 Then there exists a constant $c > 0$ which depends only on $\delta$ such that
 \begin{equation} \label{E:intersect}
  | E(\gamma, H) - E(\gamma', H) | \leq c e^{-d/2}.
 \end{equation}

\begin{proof}
 Let $\ell = d(p_1, p_2)$ be the distance along the geodesic $\gamma$, and $d = d(\mathrm{i}, p_1)$ the distance between the base point and the entry point of $\gamma$ in $H$.
 Moreover, for $i = 1, 2$ let us denote by $\alpha_i$ the arc of the horoball between $p_i$ and~$p_i'$.

\begin{figure} \label{F:dist1}
 \begin{center}
  \begin{tikzpicture}[scale=1.8]
   \draw[->, gray] (-2.3,0) -- (1,0) node[right] {};
   \draw[->] (0,0) node[below] {$0$} -- (0,3) node[above] {$\gamma$};
   \draw[->] (0.5,0) node[below] {$\delta$} -- (0.5,3) node[above] {$\gamma'$};

   \fill (-0.7,0) circle (1pt);
   \draw (-0.7,1.3) circle (1.3);
   \draw (-2,1.3) node[right] {$H$};
   
   \begin{scope}
    \clip (0,0) -- (0.5,0) -- (0.5,3) -- (0,3) -- (0,0);
    \draw[ultra thick] (-0.7,1.3) circle (1.3);
   \end{scope}
   \path (0.3,0.28) node{$\alpha_1$};
   \path (0.3,2.32) node{$\alpha_2$};
   
   \draw (0,0.25) node[left]{$p_1$};
   \draw (0,2.35) node[left]{$p_2$};
   \draw (0.5,0.8) node[right]{$p_1'$};
   \draw (0.5,1.85) node[right]{$p_2'$};
   
   \begin{scope}[xshift=3.5cm]
    \draw[->, gray] (-1,0) -- (2,0) node[right] {};
    \draw[->] (0,0) node[below] {$0$} -- (0,3) node[above] {$\gamma$};
    \draw[->] (1.5,0) node[below] {$\delta$} -- (1.5,3) node[above] {$\gamma'$};

    \draw (0.6,1.4) circle (1.2);
    \draw (-0.6,1.3) node[right] {$H$};
    
    \draw[->] (0.6,0) -- (0.6,3) node[above] {$\gamma''$};

    \begin{scope}
     \clip (0,0) -- (1.5,0) -- (1.5,3) -- (0,3) -- (0,0);
     \draw[ultra thick] (0.6,1.4) circle (1.2);
    \end{scope}
   
    \draw (0,0.3) node[left]{$p_1$};
    \draw (0,2.45) node[left]{$p_2$};
    \draw (1.5,0.6) node[right]{$p_1'$};
    \draw (1.5,2.3) node[right]{$p_2'$};
    \draw (0.6,0.4) node[right]{$p_1''$};
    \draw (0.6,2.4) node[right]{$p_2''$};
   \end{scope}
  \end{tikzpicture}
  \caption{Two situations in \autoref{L:dist1}: first with the center of $C_H$ to the left of $\gamma$ and $\gamma'$, then with the center of $C_H$ between $\gamma$ and $\gamma'$.}
  \label{fig:intersection_gamma_prime}
 \end{center}
\end{figure}
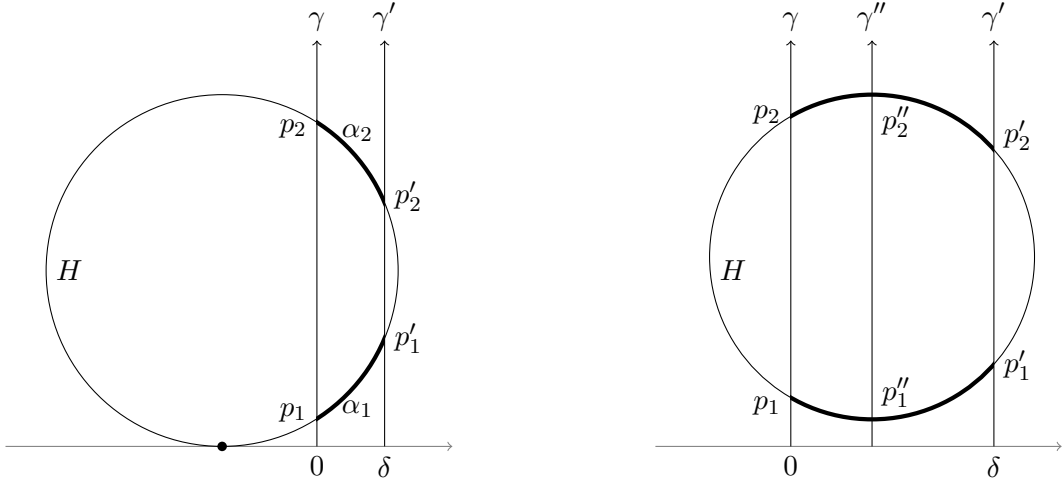

 Note that the boundary point of $H$ does not have to be contained in the same plane as $\gamma$ and $\gamma'$. Therefore, the intersection $C_H$ of $H$ with this plane might be a horoball or might be a Euclidean circle inside the hyperbolic plane.
 
 We will distinguish three cases depending on the position of the Euclidean center of $C_H$.
 We first consider the case where the center of $C_H$ is on the left of $\gamma$ and $\gamma'$.
 Second, we consider the case where the center 
 is on the right of $\gamma$ and $\gamma'$.
 The third case is when the center is between $\gamma$ and $\gamma'$.
 
 So, assume first that the center of the circle $C_H$ is on the left of $\gamma$ and $\gamma'$. Then $E(\gamma', H) < E(\gamma, H)$ and the hyperbolic length of $\alpha_1$ is bounded above by
 \begin{equation} \label{E:l-alpha1}
  \ell(\alpha_1) \leq \frac{\delta}{y_1} + \int_{y_1}^{y_1'} \frac{dy}{y}
  .
 \end{equation}
 Now, if one keeps $p_1$ fixed 
 and lets $H$ vary, the largest possible value of $y_1'$ is reached when~$H$ is tangent to $\gamma'$ and the boundary point of $H$ is most far away from $\gamma'$ which is when the boundary point of $H$ lies in the same plane as $\gamma$ and $\gamma'$. Hence, we can find a bound on the value of $y_1'$ by using \autoref{L:comp}.
 In particular, if $\gamma'$ is tangent to $H$, then by definition of $\ell$ and \autoref{L:comp}
 $$\log \left( \frac{y_2}{y_1} \right) = \ell = 2 \log \left( 1+ \sqrt{\frac{2\delta}{y_1}} \right)$$ 
 and 
 $$y_1' = y_2' = \frac{y_1 + y_2}{2},$$
 so we have 
 $$\int_{y_1}^{y_1'} \frac{dy}{y} = \log \left( \frac{y_1'}{y_1} \right) = \log \left( \frac{1 + \frac{y_2}{y_1}}{2} \right) = \log \left( \frac{1 + \left( 1 + \sqrt{\frac{2\delta}{y_1}} \right)^2 }{2} \right).$$ 
 Hence, using $y_1 = e^d$ and $\log(1+x) \leq x$, we have
 \begin{equation*}
  \log \left( \frac{y_1'}{y_1} \right)
  = \log \left( 1 + \delta e^{-d} + \sqrt{ 2\delta} e^{\sfrac{-d}{2}} \right) 
  \leq \left( \delta + \sqrt{2 \delta} \right) e^{\sfrac{-d}{2}}
  .
 \end{equation*}
 Plugging this into equation~\eqref{E:l-alpha1}, we get
 $$\ell(\alpha_1) \leq \delta e^{-d} +  \left( \delta + \sqrt{2 \delta} \right) e^{\sfrac{-d}{2}} \leq  \left( 2 \delta + \sqrt{2 \delta} \right) e^{\sfrac{-d}{2}}.$$
 Finally, let us note that, as $y_1 < y_2'$ and $|y_1 - y_1'| = |y_2 - y_2'|$,  
 $$\ell(\alpha_2) \leq \frac{\delta}{y_2'} + \int_{y_2}^{y_2'} \frac{dy}{y} \leq \frac{\delta}{y_1} + \int_{y_1}^{y_1'} \frac{dy}{y}$$
 hence we can apply the same bound as before and obtain 
 $$\ell(\alpha_1) + \ell(\alpha_2) \leq 2 \left( 2 \delta + \sqrt{2 \delta} \right) e^{\sfrac{-d}{2}}.$$
 Now, let $\beta$ be the shortest path from $p_1'$ to $p_2'$ along the boundary of $H$. By definition, the length of~$\beta$ is $E(\gamma', H)$. 
 On the other hand, by definition, the excursion $E(\gamma, H)$ is the length of the shortest path from $p_1$ to $p_2$ along the boundary of $H$.
 Note that also the union of $\alpha_1$, $\beta$, and $\alpha_2$ is a path from $p_1$ to $p_2$ along the boundary of $H$.
 Therefore, we have
 $$E(\gamma, H) \leq  E(\gamma', H) + \ell(\alpha_1) + \ell(\alpha_2) \leq E(\gamma', H) + 2 \left( 2 \delta + \sqrt{2 \delta} \right) e^{\sfrac{-d}{2}},$$
 which shows the desired bound for the case when the Euclidean center of the circle $C_H$ is on the left of~$\gamma$ and $\gamma'$.
 
 Consider now the case when the center of $C_H$ is on the right of $\gamma$ and $\gamma'$.
 From the previous case, we get by symmetry arguments that we have
 \begin{equation*}
  | E(\gamma', H) - E(\gamma, H) | \leq 2 \left( 2 \delta + \sqrt{2 \delta} \right) e^{ \sfrac{-d'}{2} }
 \end{equation*}
 where $d' \coloneqq d(\mathrm{i}, p_1')$.
 Furthermore, we have
 \begin{equation*}
  d' \geq d - \log \left( \frac{y_1'}{y_1} \right)
  \geq d - \left( \delta + \sqrt{2 \delta} \right) e^{\sfrac{-d}{2}}
  \geq d - \left( \delta + \sqrt{2 \delta} \right)
 \end{equation*}
 which implies
 \begin{equation*}
  2 \left( 2 \delta + \sqrt{2 \delta} \right) e^{ \sfrac{-d'}{2} }
  \leq 2 \left( 2 \delta + \sqrt{2 \delta} \right) e^{ \delta + \sqrt{2 \delta} } \cdot  e^{ \sfrac{-d}{2}}
  .
 \end{equation*}
 Choosing $c \coloneqq 2 \left( 2 \delta + \sqrt{2 \delta} \right) e^{\delta + \sqrt{2 \delta}}$ shows the desired bound for the second case.
 
 In the third case, the center of the circle $C_H$ is between $\gamma$ and $\gamma'$.
 Let $\gamma''$ be a geodesic in the plane defined by $\gamma$ and $\gamma'$, that contains this center and is parallel to $\gamma$ and $\gamma'$ (see \autoref{fig:intersection_gamma_prime}).
 Let $p_1''$ and~$p_2''$ be the intersection points of $H$ with $\gamma''$ and $y_1''$ and $y_2''$ analogously to before. Then we have $y_1'' \geq 1$ as $y_1 \geq 1+\delta$, $y_2 \geq 1+\delta$, and the radius of the Euclidean circle $C_H$ is at least $\frac{\delta}{2}$.
 
 Note that in the previous two cases, we have not used $y_1 \geq 1+\delta$ and $y_1' \geq 1+\delta$ but only $y_1 \geq 1$ and $y_1' \geq 1$.
 Therefore, we can apply the cases from before to compare $E(\gamma, H)$ with $E(\gamma'', H)$ and $E(\gamma', H)$ with $E(\gamma'', H)$. This gives us
 \begin{equation*}
  | E(\gamma, H) - E(\gamma', H) |
  \leq | E(\gamma, H) - E(\gamma'', H) | + | E(\gamma', H) - E(\gamma'', H) |
  \leq c e^{\sfrac{-d}{2}} + c e^{\sfrac{-d'}{2}}
  .
 \end{equation*}
 As before, we can express $d'$ by $d$ and obtain a uniform constant which depends only on $\delta$.
\end{proof}
\end{lem}

We now use \autoref{lem:excursion_comparison_non-intersecting_geodesics} to prove a statement about the sum of the \kth powers of the excursions of a geodesic ray in a collection of horoballs. This will be an ingredient in the proof of \autoref{prop:comparison_twosided_with_same_endpoint}.

\begin{lem} \label{lem_disjoint_horoballs_have_bounded_total_extension}
 Consider the hyperbolic plane $\mathbb{H}^2$ inside the hyperbolic $N$--space $\mathbb{H}^N$ with coordinates~$x, y$.
 Let $(H_i)_{i \geq 1}$ be a collection of disjoint horoballs which intersect $\gamma \coloneqq \{x = 0, y\geq 1\}$ in two points and do not intersect $\gamma' \coloneqq \{ x= \delta\}$.
 
 Then for every $k > 1$, there exists a universal constant $c > 0$, which depends only on $\delta$ and $k$, such that
 \begin{equation*}
  \sum_{i = 1}^\infty E(\gamma, H_i)^k \leq c
  .
 \end{equation*}

\begin{proof}
 We can choose the order of the horoballs such that for every $i \geq 1$ we have that $H_i$ intersects $\gamma \coloneqq \{x = 0\}$ in $p_1^{(i)} = \mathrm{i}y_1^{(i)}$ and $p_2^{(i)} = \mathrm{i}y_2^{(i)}$ and we have $1 \leq y_1^{(1)} < y_2^{(1)} \leq y_1^{(2)} < \ldots$.
 
 Then by \autoref{lem:excursion_comparison_non-intersecting_geodesics} and because the horoballs are disjoint, we have
 \begin{equation*}
  \sum_{i = 1}^\infty E(\gamma, H_i)^k
  \leq c \sum_{i = 1}^\infty \int_{y_1^{(i)}}^{y_2^{(i)}} \frac{dy}{y^{\frac{k+1}{2}}}
  \leq  c \int_1^\infty \frac{dy}{y^{\frac{k+1}{2}}} = \frac{2c}{k-1}
  < +\infty
  .
 \end{equation*}
\end{proof}
\end{lem}

We finish this section with two lemmas that will be helpful in the proof of \autoref{prop:comparison_twosided_with_same_endpoint}.

\begin{lem} \label{L:recurrence}
Let $0 < c < 2$ be a constant, and let $(d_i)$ be a sequence such that $d_0 \geq 0$, and 
$$d_{i+1} \geq d_i + c e^{- \sfrac{d_i}{2} } \qquad \text{for all }i \in \mathbb{N}.$$
Then for each $i \in \mathbb{N}$ we have 
$$d_i \geq 2 \log \left( 1+ \frac{i c}{2} \right).$$
\end{lem}

\begin{proof}
Consider the function $f(x) \coloneqq x + c e^{\sfrac{-x}{2} }$. Since $f'(x)= 1-\frac{c}{2} \cdot e^{\sfrac{-x}{2}}$ and $c < 2$, we have that $f(x)$ is increasing for $x \geq 0$.

Let us now prove the claim by induction. For this, suppose that $d_i \geq 2 \log\left( 1 + \frac{i c}{2}\right)$ for an $i\in \mathbb{N}$. Then 
$$d_{i + 1} \geq f(d_i) \geq f\left(2 \log \left( 1 + \frac{ic}{2} \right) \right) = 2 \log \left( 1 + \frac{ic}{2}\right) + \frac{c}{1 + \frac{ic}{2}} \geq 2 \log \left( 1 + \frac{(i+1) c}{2}\right)$$
where the last inequality follows from 
\begin{equation*}
 2 \log \left( 1 + \frac{(i+1) c}{2} \right) - 2 \log\left( 1 + \frac{ic}{2} \right)
 = 2 \log \left(\frac{1 + \frac{(i+1) c}{2}}{1 + \frac{ic}{2}}\right)
 = 2 \log \left( 1 + \frac{\frac{c}{2}}{1 + \frac{ic}{2}} \right)
 \leq \frac{c}{1 + \frac{ic}{2}}
 .
\end{equation*}
\end{proof}

\begin{lem} \label{lem:powers_k}
 Let $k \geq 1$, $k\in \mathbb{R}$ and $x, y \geq 0$. Then we have
 \begin{equation*}
  (x+y)^{k} \leq x^{k} + 2^{k-1} k\, y \left( x^{k-1} + y^{k-1} \right)
  .
 \end{equation*}
 
\begin{proof}
 Note that
 \begin{equation*}
  (x+y)^{k-1}
  \leq \left( 2\cdot \max\{x, y\} \right)^{k-1}
  = 2^{k-1} \cdot \max\{x^{k-1}, y^{k-1}\}
  \leq 2^{k-1} (x^{k-1} + y^{k-1})
  .
 \end{equation*}
 As the function $t^{k-1}$ is non-decreasing on $\mathbb{R}_{\geq 0}$, we can calculate
 \begin{equation*}
  (x+y)^k
  = x^k + \int_{x}^{x+y} k\, t^{k-1}\, dt
  \leq x^k + y \cdot k (x+y)^{k-1}
  \leq x^k + 2^{k-1} k\, y (x^{k-1} + y^{k-1})
  .
 \end{equation*}
\end{proof}
\end{lem}

\subsection{Background on random walks} \label{sec:background_random_walks}

Let $\Gamma < \Isom(\mathbb{H}^N)$ be a countable group of isometries of hyperbolic $N$--space.
We now recall some basic definitions in order to define a random walk in this setting.

Let $\mu$ be a measure on $\Gamma$ which is \emph{generating}, that is, $\Gamma$ is generated by the support of $\mu$ as a semi-group.
The \emph{step space} is the measure space $(\Gamma^\mathbb{N}, \mu^\mathbb{N})$, which is the space of increments for the random walk.
Its elements are denoted by $(g_n)$. 
Let us consider the map $\pi: \Gamma^\mathbb{N} \to \Gamma^\mathbb{N}$ defined by $\pi((g_n)) = (w_n)$ with 
$$w_n \coloneqq g_1 \dots g_n$$
which associates to each sequence of increments the corresponding positions of the random walk on~$\Gamma$. 
The \emph{space of sample paths} is $\Omega = (\Gamma^\mathbb{N}, \mathbb{P})$, where $\mathbb{P} = \pi_\star(\mu^\mathbb{N})$
is the pushforward of the product measure to $\Omega$. Elements of $\Omega$ are denoted by $\omega = (w_n)$.

Picking a base point $x \in \mathbb{H}^N$ defines a sequence $(w_n x)$ in $\mathbb{H}^N$ for every $(w_n) \in \Omega$ with
$$w_n x = g_1 \dots g_n x.$$
We call this process a \emph{random walk} on $\mathbb{H}^N$ driven by $\mu$.

We recall now some properties of random walks which are used in the following section.
Recall that a group of isometries of $\mathbb{H}^N$ is \emph{non-elementary} if it contains at least two loxodromic elements with disjoint fixed sets on the boundary of $\mathbb{H}^N$. 

\begin{prop}[Convergence to the boundary \cite{furstenberg}]
If the group $\Gamma$ is non-elementary, then for almost every sample path~$\omega = (w_n) \in \Omega$ and every~$x\in \mathbb{H}^N$, there exists the limit
$$\xi^+(\omega) \coloneqq \lim_{n \to \infty} w_n x \in \partial \mathbb{H}^N.$$
\end{prop}

Since the random walk converges almost surely, we can define its \emph{hitting measure} as the probability measure~$\nu$ on $\partial \mathbb{H}^N$ 
such that
$$\nu(A) \coloneqq \mathbb{P}( \xi^+(\omega) \in A )$$
for any Borel set $A \subseteq \partial \mathbb{H}^N$.

The random walk is said to have \emph{finite \kth moment} in a metric $d$ on $\mathbb{H}^N$ if $\int d(x, gx)^k \ d\mu(g) < + \infty$ for some $x\in \mathbb{H}^N$, and \emph{finite exponential moment} if there exists $\alpha > 0$ for which $\int e^{\alpha d(x, gx)} \ d\mu(g) <~+\infty$. 
As an immediate application of Kingman's ergodic theorem \cite{kingman_68}, one gets:

\begin{prop}[Linear drift] \label{prop:linear_drift}
If the random walk has finite first moment in a metric $d$ on $\mathbb{H}^N$, then there exists $L \geq 0$ such that for almost every sample path $\omega \in \Omega$,
$$\lim_{n \to \infty} \frac{d(x, w_n x)}{n} = L.$$
\end{prop}

Moreover, if $\Gamma$ is non-elementary and $d$ is the hyperbolic metric, then $L > 0$ (\cite{furstenberg}, \cite{maher_tiozzo_18}). 
Furthermore, in hyperbolic spaces, random walks track hyperbolic geodesics quite closely: in particular, they track them up to sublinear error in the number of steps.

\begin{prop}[Sublinear tracking \cite{Kai-hyperbolic}] \label{prop:sublinear_tracking}
If the random walk has finite first moment in the hyperbolic metric $d$,
then for almost every sample path $\omega$ and $\gamma \colon [0, \infty) \to \mathbb{H}^N$ the geodesic ray connecting $x$ and $\xi^+(\omega)$, we have
$$\lim_{n \to \infty} \frac{d(w_n x, \gamma(L n))}{n} = 0.$$
\end{prop}

In the next section, we consider also the \emph{two-sided random walk}. That is, the map $\pi \colon \Gamma^\mathbb{N} \to \Gamma^\mathbb{N}$ can be extended 
to a map $\overline{\pi} : \Gamma^\mathbb{Z} \to \Gamma^\mathbb{Z}$ by 
$$w_n \coloneqq \left\{ 
\begin{array}{ll} g_1\dots g_n & \textup{if }n \geq 1 \\
1 & \textup{if } n = 0 \\
g_0^{-1} g_{-1}^{-1} \dots g_{n+1}^{-1} & \textup{if }n \leq -1 \end{array} \right.$$
and the distribution of the two-sided random walk is given by the measure $\overline{\mathbb{P}} \coloneqq \overline{\pi}_\star(\mu^\mathbb{Z})$
on the space $\overline{\Omega} = \Gamma^\mathbb{Z}$. Let us denote by $\sigma : \Gamma^\mathbb{Z} \to \Gamma^\mathbb{Z}$
the shift on the space of increments, that is we have $(\sigma(g_n))_n \coloneqq g_{n+1}$. This map preserves the measure $\mu^\mathbb{Z}$. 

If the group $\Gamma$ is non-elementary, then both the forward and backward random walks converge almost surely to the boundary of $\mathbb{H}^N$, 
defining 
$$ \xi^+(\omega) \coloneqq \lim_{n \to \infty} w_n x, \qquad \xi^-(\omega) \coloneqq \lim_{n \to - \infty} w_n x.$$
We also define the bi-infinite geodesic $\gamma_\omega$ connecting $\xi^-(\omega)$ and $\xi^+(\omega)$. 
If $\sigma$ is the shift on the space of increments, we have by definition
$$\gamma_{\sigma^n \omega} = [w_n^{-1} \xi^-(\omega), w_n^{-1} \xi^+(\omega)] = w_n^{-1} \gamma_\omega.$$
We denote by $p_j(\omega)$ the closest point projection of $w_j x$ onto $\gamma_\omega$. 
Then for any $n \in \mathbb{Z}$, we have 
$$p_j(\sigma^n \omega) = \textup{proj}(g_{n+1} \dots g_{n+j} x, w_n^{-1} \gamma_\omega) = w_n^{-1} \textup{proj}(w_{n+j} x, \gamma_\omega) = 
w_n^{-1} p_{n+j}(\omega).$$

\section{Excursion for random walks} \label{sec:excursion_random_walks}

We are now ready to define the excursion of a bi-infinite sample path in the thin part. 

First fix an $N \geq 2$ and let $\Gamma$ be a discrete subgroup of $\Isom(\mathbb{H}^N)$ such that $M \coloneqq \mathbb{H}^N/\Gamma$ is a non-compact hyperbolic orbifold of finite volume.
Note that such a group $\Gamma$ is always finitely generated and non-elementary. Moreover, the quotient $M$ is the union of a compact part and finitely many cusps.
For each cusp of $M$, we choose a neighborhood in $M$ whose lifts are horoballs in $\mathbb{H}^N$. This gives us a $\Gamma$--invariant collection $\mathcal{H}$ of horoballs.
We can choose the neighborhoods sufficiently small such that the horoballs are pairwise disjoint and even such that they have pairwise distance at least some~$R > 4\delta$.
We refer to the complement of the cusps as the \emph{thick part} of $M$ and to the complement of~$\mathcal{H}$ as the \emph{thick part} of $\mathbb{H}^N$.

We shall also need the notion of \emph{relative metric} on $\mathbb{H}^N$, obtained by ``shrinking'' the horoballs to have finite diameter. 
Let us first define 
$$\widetilde{d}(x, y) := \left\{ \begin{array}{ll} \max\{ d(x, y), 1 \} & \textup{if }x, y \textup{ belong to the same horoball} \\
d(x, y) & \textup{otherwise} \end{array}\right.$$
and then the \emph{relative metric}  $\drel$ is the path metric induced by $\widetilde{d}$, i.e.\
$\drel(x, y) := \inf \sum_{i = 0}^{r-1} \widetilde{d}(x_i, x_{i+1})$
where the inf is taken over all chains $x = x_0, \dots, x_r = y$ of points in $\mathbb{H}^N$. 

By definition, $\drel(x, y) \leq d(x, y)$, and the diameter of each horoball in $\drel$ is at most $1$. 
Moreover, the space $(\mathbb{H}^N, \drel)$ is a non-proper, $\delta$--hyperbolic space on which $\Gamma$ acts by isometries. 
This space is quasi-isometric to the \emph{coned-off} Cayley graph of the relatively hyperbolic group $\Gamma$; see \cite[Section 4.2]{Farb} for details. 

\medskip
Now let $\mu$ be a generating probability measure on $\Gamma$ and pick $x\in \mathbb{H}^N$ to be a base point which is contained in the thick part.
As in \autoref{sec:background_random_walks}, we have the set $\overline{\Omega} \coloneqq \Gamma^\mathbb{Z}$ of bi-infinite sample paths with the measure $\overline{\mathbb{P}}$, the bi-infinite geodesic $\gamma_\omega \coloneqq [\xi^-(\omega), \xi^+(\omega)]$ and the closest point projection $p_n(\omega)$ of $w_n x$ to~$\gamma_\omega$.
Furthermore, if $\gamma$ is a geodesic, and $H \in \mathcal{H}$ a horoball which intersects $\gamma$, we denote as $p_H$ the midpoint of the intersection of $\gamma$ and $H$ (see \autoref{fig:excursion}).

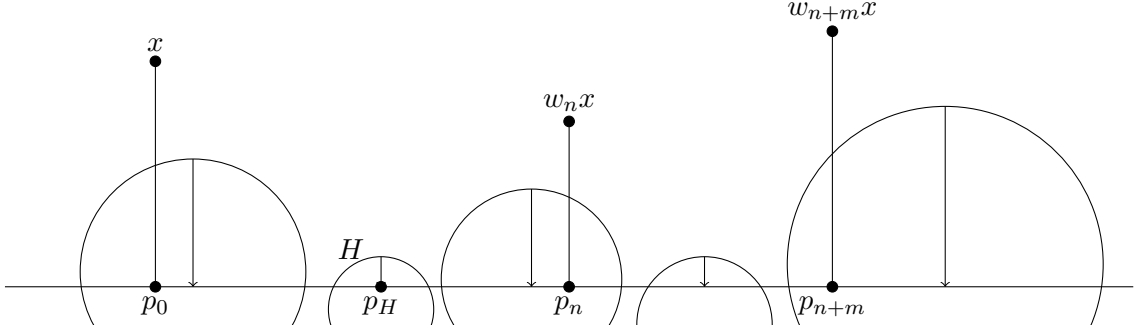
\begin{figure}
 \centering
 \begin{tikzpicture}
  \path[clip] (0,-0.5) -- (15,-0.5) -- (15,5) -- (0,5);
  \draw (0,0) -- (15,0);
  \draw (2.5,0.2) circle (1.5cm);
  \draw (5,-0.3) circle (0.7cm);
  \draw (7,0.1) circle (1.2cm);
  \draw (9.3,-0.5) circle (0.9cm);
  \draw (12.5,0.3) circle (2.1cm);
  \draw[fill] (2,3) node[above]{$x$} circle (2pt);
  \draw (2,3) -- (2,0);
  \draw[fill] (2,0) node[below]{$p_0$} circle (2pt);
  \draw[fill] (7.5,2.2) node[above]{$w_n x$} circle (2pt);
  \draw (7.5,2.2) -- (7.5,0);
  \draw[fill] (7.5,0) node[below]{$p_n$} circle (2pt);
  \draw[fill] (11,3.4) node[above]{$w_{n+m} x$} circle (2pt);
  \draw (11,3.4) -- (11,0);
  \draw[fill] (11,0) node[below]{$p_{n+m}$} circle (2pt);
  \draw[->] (5,0.4) -- (5,0);
  \draw[fill] (5,0) node[below]{$p_H$} circle (2pt);
  \path (4.6,0.5) node{$H$};
  \draw[->] (2.5,1.7) -- (2.5,0);
  \draw[->] (7,1.3) -- (7,0);
  \draw[->] (9.3,0.4) -- (9.3,0);
  \draw[->] (12.5,2.4) -- (12.5,0);
 \end{tikzpicture}
 \label{fig:excursion}
 \caption{Setting for the definition of excursion (\autoref{D:exc}).}
\end{figure}

To define the \kth excursion of the bi-infinite sample path $\omega$ in the horoballs up to step $n$, we want to consider all the excursions of $\gamma_\omega$ in all horoballs $H \in \mathcal{H}$ while making sure that every horoball is considered only once along $\gamma_\omega$. 
To do this, let us denote by $\mathcal{H}_{n, \omega}$ for every $n\geq 0$ the set of horoballs~$H \in \mathcal{H}$ such that the midpoint $p_H$ of the intersection lies between $p_0(\omega)$ and $p_n(\omega)$. 

\begin{definition}[Excursion of bi-infinite sample path] \label{D:exc}
Let $k \geq 1$ be a real number. For each $n \geq 0$, we define the \emph{\kth excursion of the bi-infinite sample path $\omega \in \overline{\Omega}$ up to step $n$} as 
$$X_n^{(k)}(\omega) \coloneqq \sum_{H \in \mathcal{H}_{n, \omega}} E(\gamma_\omega, H)^k.$$
\end{definition}

\subsection{Step average and time average of excursion}

The following proposition shows that the step average of the \kth excursion is well-defined for almost all bi-infinite sample paths.

\pagebreak

\begin{prop}[Step average of excursion exists and is finite] \label{P:limit-Xn}
Let $k' > k \geq 1$, and suppose that $\mu$ has finite $(k')^{th}$ moment in some word metric on $\Gamma$ and finite exponential moment in the relative metric. Then for almost every $\omega \in \overline{\Omega}$ the limit 
$$\lim_{n \to \infty} \frac{X_n^{(k)}(\omega)}{n}$$
exists and is finite. 
\end{prop}

To prove the proposition, we use Kingman's subadditive ergodic theorem \cite{kingman_68} and \autoref{P:finite-moment}, whose proof uses the thick distance.

To define the thick distance, let $x,y$ be two points in the thick part of $\mathbb{H}^N$ and $[x,y]$ the geodesic segment between them.
Furthermore, let $H_1,\ldots,H_n \in \mathcal{H}$ be the horoballs, that intersect $[x,y]$. 
Let $q_i$ and $q_i'$ be the entry and exit points in the horoball $H_i$ and order them such that $[x,y]$ goes through $x,q_1,q_1',\ldots, q_n, q_n', y$ in this order.
Then the \emph{thick distance} $\dthick$ of $x$ and $y$ is defined as
\begin{equation*}
 \dthick (x,y) = d(x, q_1) + d_{\partial H_1}(q_1,q_1') + d(q_1',q_2) + \ldots + d_{\partial H_n}(q_n,q_n') + d(q_n', y)
\end{equation*}
where $d_{\partial H_i}$ is the path metric on $\partial H_i$ induced by the hyperbolic metric $d$.

The thick distance is comparable to the word metric in the following sense (see~\cite[Lemma~2.1]{gadre_maher_tiozzo_15} for a proof).

\begin{lem} \label{lem:thick_distance_word_metric}
 For any finite generating set $S$ of $\Gamma$ and any point $x$ in the thick part of $\mathbb{H}^N$, there exists a constant $C > 0$ such that for every $g\in \Gamma$, we have
 \begin{equation*}
  \frac{1}{C} \left\| g \right\|_S - C \leq \dthick(x, gx) \leq C \left\| g \right\|_S + C
 \end{equation*}
 where $\left\| . \right\|_S$ denotes the word metric on $\Gamma$ with respect to $S$.
\end{lem}

The following is the key ingredient for the proof of \autoref{P:limit-Xn}.

\begin{prop} \label{P:finite-moment}
If $k' > k \geq 1$ and $\mu$ has finite $(k')^{th}$-moment with respect to some word metric and has finite exponential moment with respect to $\drel$, then $X_n^{(k)} \in L^1$ for any $n \geq 0$. 
\end{prop}

The first step in the proof of \autoref{P:finite-moment} is the following geometric estimate. 

\begin{lem} \label{L:moment}
Let $k\geq 1$. Then there exists $C > 0$ such that for every $\omega \in \overline{\Omega}$ with $p_0(\omega)$ and $p_n(\omega)$ not in the same horoball, we have
$$ X_n^{(k)}(\omega) \leq C  \Vert w_n(\omega) \Vert^k  + C.$$
\end{lem}

\begin{proof}
Let $\mathcal{H}_{n, \omega} = \{H_1, \ldots, H_r\}$, that is, $H_1, \dots, H_r$ are the horoballs which intersect $[p_0, p_n]$ and such that $p_{H_i} \in [p_0, p_n]$ (see \autoref{fig:cases_gamma_segment}). 

If $d(p_0, p_n) \leq 4 \delta$, then \autoref{L:comp} implies $E(\gamma, H_i) \leq 2 \sinh(4 \delta)$ for each $i$, and $r \leq \frac{4 \delta}{R} + 1$
since horoballs have minimal distance $R>0$ between them, hence $X_n^{(k)} \leq (\frac{4 \delta}{R}  +1) (2 \sinh(4 \delta))^k$ is bounded. 

Otherwise, $d(p_0, p_n) \geq 4 \delta$. Then, 
by $\delta$--hyperbolicity of $\mathbb{H}^N$, the geodesic segment $\widetilde{\gamma} \coloneqq [x, w_n x]$ passes through a $4 \delta$--neighborhood 
of both $p_0$ and $p_n$, hence every point on the segment $[p_0, p_n]$ is uniformly close to $\widetilde{\gamma}$.  

\medskip
Now, we can find a constant $C$, which depends only on $\delta$ and which can be chosen with $C \geq \delta$, 
such that the following are true for any $i = 1, \dots, r$:
\begin{enumerate}
\item
If $H_i$ does not intersect the geodesic segment $\widetilde{\gamma} \coloneqq [x, w_n x]$, then $E(\gamma, H_i) \leq C$.
This is because the geodesic segment between the entry and the exit point of $\gamma$ lies within distance~$2\delta$ of $\widetilde{\gamma}$ and hence of the thick part. Thus, the excursion is bounded.
\item
If $H_i$ intersects $\widetilde{\gamma}$ and both the entry and exit points of $\gamma$ in $H_i$ belong to $[p_0, p_n]$, then $E(\gamma, H_i) \leq E(\widetilde{\gamma}, H_i) + C$ as proven in \autoref{L:dist1}.
\item 
The third case is that $H_i$ intersects $\widetilde{\gamma}$, but only one of its endpoints (i.e.\ the entry or the exit point) 
belongs to $[p_0, p_n]$, then $$E(\gamma, H_i) \leq 2 E(\widetilde{\gamma}, H_i) + C.$$
This bound follows since the midpoint $p_H$ of the projection belongs to $[p_0, p_n]$ and we can still use an argument similar to (ii). 
\end{enumerate}

\begin{figure}
 \centering
 \begin{tikzpicture}
  \path[clip] (-1,-0.5) -- (10,-0.5) -- (10,4.5) -- (-1,4.5);
  \draw[fill] (0,0) node[below]{$p_0$} circle (2pt);
  \draw[fill] (8,0) node[below]{$p_n$} circle (2pt);
  \draw (-1,0) -- (10,0);
  \draw[fill] (0,3) node[above]{$x$} circle (2pt);
  \draw (0,3) -- (0,0);
  \draw[fill] (8,4) node[above]{$w_n x$} circle (2pt);
  \draw (8,4) -- (8,0);
  \draw (0,3) .. controls +(-80:1cm) and +(150:1cm) ..
  (1,0.5) .. controls +(-30:1cm) and +(-160:1.2cm) .. 
  (7,0.6) .. controls +(20:1.2cm) and +(-95:1cm)  .. (8,4);
  \draw (2,-0.3) circle (1cm);
  \draw (4.5,-0.8) circle (0.9cm);
  \draw (7.5,0.2) circle (1.5cm);
 \end{tikzpicture}
 \label{fig:cases_gamma_segment}
 \caption{The three cases in the proof of \autoref{L:moment}. From left to right, the three horoballs correspond to cases (ii), (i), and (iii).}
\end{figure}
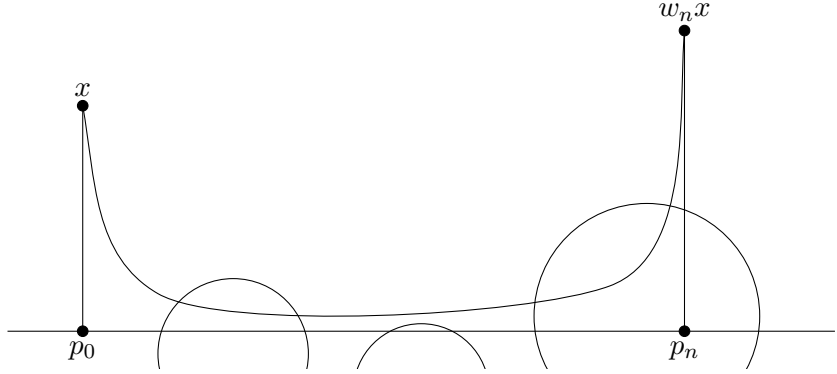

Now, let us denote as $a$ the exit point of $H_1$, and as $b$ the entry point of $H_r$, and as $a'$ the closest point projection of $a$ to $\widetilde{\gamma}$ in the thick part, and $b'$ the closest point projection of $b$ to $\widetilde{\gamma}$ in the thick part.
Since there exists a minimal distance $R > 0$ between any two horoballs of our collection, we have
$$(r-1)R \leq \dthick(a, b) \leq \dthick(a', b') + 2\max\{C, 2\delta\} \leq \dthick(x, w_n x) + 2\max\{C,2\delta\}$$
thus, 
\begin{equation} \label{E:r}
r \leq \frac{\dthick(x, w_n x)}{R} + C_1
\end{equation}
with $C_1 =  \frac{2\max\{C,2\delta\} + R}{R}$. Now, we can estimate $X_n^{(k)}$:  by using (i), (ii), and (iii) we have
$$X_n^{(k)} = \sum_{i = 1}^r E(\gamma, H_i)^k \leq \sum_{i = 1}^r \left( 2 E(\widetilde{\gamma}, H_i) + C \right)^k \leq \left( 2 \sum_{i = 1}^r E(\widetilde{\gamma}, H_i) + Cr \right)^k  \leq$$
and using that $\sum_{i = 1}^r E(\widetilde{\gamma}, H_i) \leq \dthick(x, w_n x)$ and equation \eqref{E:r}
$$\leq \left( \left( 2 + \frac{C}{R}\right) \dthick(x, w_n x) + C C_1\right)^k \leq$$
and using $\dthick(x, w_n x) \leq C_2 \Vert w_n \Vert + C_2$ (see \autoref{lem:thick_distance_word_metric}),
$$\leq \left( C_3 \Vert w_n \Vert + C_4 \right)^k $$
where $C_3 = \left( 2 + \frac{C}{R} \right) C_2$ and $C_4 = C_2\left( 2 + \frac{C}{R}\right) + C C_1$, hence (since $(x+y)^k \leq 2^k (x^k + y^k)$)
$$\leq 2^k C_3^k \Vert w_n \Vert^k + 2^k C_4^k$$
yielding the claim. 
\end{proof}

We now turn to the proof of \autoref{P:finite-moment}. Note that, if $p_0$ and $p_n$ do not belong to the same horoball, then \autoref{L:moment} yields an upper bound on $X^{(k)}_n$. 
However, it may happen that $p_0$ and $p_n$ belong to the same horoball; to consider the general case, define 
the return time $\tau : \overline{\Omega} \to \mathbb{N}$~as 
$$\tau(\omega) := \inf\{ k \geq 0 \ : \ p_k \textup{ lies in }[p_0, \xi^+)\textup{ and does not belong to the same horoball as }p_0 \}.$$
Since by construction~$p_{\tau}$ and $p_0$ do not lie in the same horoball, from \autoref{L:moment} we obtain: 
$$X^{(k)}_n(\omega) \leq \left\{ \begin{array}{ll} C \Vert w_{\tau(\omega)}(\omega) \Vert^k + C &\textup{if }p_0, p_n \textup{ belong to the same horoball} \\
C \Vert w_{n}(\omega) \Vert^k + C & \textup{otherwise}. 
\end{array} \right.$$
Therefore, it is enough to show
\begin{equation} \label{E:moment}
\int  \Vert  w_{\tau(\omega)} (\omega)  \Vert^k \ d \mathbb{P}(\omega) < + \infty
\end{equation}
to prove \autoref{P:finite-moment}.

We use the following exponential decay for the return time. 

\begin{lem} \label{L:tau-decay}
Suppose that $\mu$ has finite exponential moment with respect to $\drel$. Then there exist~$C, \alpha > 0$ such that 
$$\mathbb{P}(\omega \in \overline{\Omega} \ : \ \tau(\omega) \geq n ) \leq C e^{- \alpha n}$$
for all $n$.
\end{lem}

\begin{proof}
Let us consider the action of $\Gamma$ on the non-proper, $\delta$--hyperbolic space $(\mathbb{H}^N, \drel)$. 
Note that geodesics for the hyperbolic metric are unparameterized quasi-geodesics for the relative metric; moreover, there exists $D > 0$
 such that, with respect to $\drel$, any closest point projection of $w_n x$ to $[\xi^-, \xi^+]$ lies within distance $D$ of $p_n$.
 
First of all, by \cite[Theorem 1]{Sunderland}, there exist $L_1, C_1, \alpha_1 > 0$ such that 
\begin{equation} \label{E:exp-drift}
\mathbb{P}(\drel(x, w_n x) \leq L_1 n) \leq C_1 e^{- \alpha_1 n}\qquad \textup{for all }n \geq 0.
\end{equation}
Moreover, by \cite[Proposition 8.2]{BMSS},  there exist $C_2, \alpha_2 > 0$ such that 
\begin{equation} \label{E:BMSS}
\mathbb{P}( \drel(w_i x, [x, w_j x]) \geq R ) \leq C_2 e^{- \alpha_2 R}
\end{equation}
for any $1 \leq i \leq j$ and any $R \geq 0$. 
Applying \eqref{E:BMSS} with $i = n, j = 2n$ and using that the action is by isometries and the shift in the step space is measure-preserving, 
we get 
$$\mathbb{P}( \drel(x, [w_{-n} x, w_{n} x]) \geq R ) = \mathbb{P}( \drel(w_n x, [x, w_{ 2 n} x]) \geq R ) \leq C_2 e^{- \alpha_2 R}.$$
Hence, by taking the limit as $n \to \infty$, 
$$\mathbb{P}( \drel(x, p_0) \geq R ) \leq \mathbb{P}( \drel(x, [\xi^-, \xi^+]) \geq R - D) \leq C_2 e^{- \alpha_2 (R - D)}$$
and also, since the shift in the step space is measure-preserving, for any $n$, 
$$\mathbb{P}( \drel(w_n x, p_n) \geq R ) \leq \mathbb{P}( \drel(w_n x, [\xi^-, \xi^+]) \geq R - D) \leq C_2 e^{- \alpha_2 (R - D)}.$$
Thus, by setting $R = L_1 n/3$ and using the triangle inequality, we have 
\begin{align} \label{E:one-triang}
\drel(p_0, p_n) & \geq \drel (x, w_n x) - \drel(x, p_0) - \drel (w_n x, p_n)  \nonumber \\
&  \geq  L_1 n - \frac{1}{3} L_1 n - \frac{1}{3} L_1 n  = \frac{1}{3} L_1 n  > 1  
 \end{align}
for $n$ large enough, with probability at least $1 -  \left( 2 C_2 e^{\alpha_2 D} e^{-\frac{\alpha_2 L_1}{3} n } + C_1 e^{-\alpha_1 n} \right)$. 
Note that by definition of $\drel$, if $p_0$ and $p_n$ lie in the same horoball, then $\drel(p_0, p_n) \leq 1$, hence \eqref{E:one-triang}
shows that with the above probability $p_0$ and $p_n$ do not belong to the same horoball. 

Additionally, we show now that $p_n \in [p_0, \xi^+)$ with high probability. Note that there exists $D'$, depending on $\delta$, such that 
$p_n \in (\xi^-, p_0]$ implies 
$$\drel(w_n x, [x, \xi^+)) \geq \drel(p_0, p_n) - D' \geq \frac{L_1 n}{3} - D' \geq \frac{L_1 n}{4}$$ 
for $n$ large enough.
Moreover, by setting $i = n$, $R = \frac{L_1 n}{4}$ in \eqref{E:BMSS} and taking the limit as $j \to \infty$, we have
\begin{equation*}
\mathbb{P}(\drel(w_n x, [x, \xi^+)) \geq L_1 n / 4) \leq C_2 e^{- \frac{\alpha_2 L_1}{4} n }.
\end{equation*}

Hence, we obtain that, with probability at least 
$1 -  \left( 2 C_2 e^{\alpha_2 D} e^{-\frac{\alpha_2 L_1}{3} n } + C_1 e^{-\alpha_1 n} + C_2  e^{- \frac{\alpha_2 L_1}{4} n} \right)$, 
the point $p_n$ belongs to $[p_0, \xi^+)$ and $p_0$ and $p_n$ do not belong to the same horoball.
This completes the proof of the lemma. 
\end{proof}

\begin{proof}[Proof of \autoref{P:finite-moment}]
Let $Y_i := \Vert g_i \Vert$. We claim that there exist $C_3, \alpha_3 > 0$ such that, for any indices $i_1 \leq i_2 \leq \dots \leq i_k$, we have
\begin{equation} \label{E:product}
\int_{\{ \tau \geq i_k \}} Y_{i_1} Y_{i_2} \dots Y_{i_k}  \ d \mathbb{P} \leq C_3 e^{- \alpha_3  (i_1 + \dots + i_k)}. 
\end{equation}
To prove the claim, let $p > 1$ such that $k p = k'$, and $q$ such that $\frac{1}{p } + \frac{1}{q} = 1$. Then by H\"older's inequality
\begin{align*}
\int_{\{ \tau \geq i_k \}} Y_{i_1} Y_{i_2} \dots Y_{i_k} \ d \mathbb{P} & \leq \left(  \int (Y_{i_1} Y_{i_2} \dots Y_{i_k})^p \ d \mathbb{P} \right)^{1/p} \left( \int \mathbf{1}_{\{ \tau \geq i_k \}} \ d \mathbb{P}\right)^{1/q} \\
\intertext{and, comparing the geometric and arithmetic means, }
& \leq  \left(  \int \frac{Y_{i_1}^{kp} + Y_{i_2}^{kp} +  \dots + Y_{i_k}^{kp}}{k} \ d \mathbb{P} \right)^{1/p}
 \mathbb{P}( \tau \geq i_k )^{1/q}  \\
& \leq  (\mathbb{E}[Y_1^{kp}])^{1/p} \ \mathbb{P}( \tau \geq i_k )^{1/q} \\
\intertext{and by \autoref{L:tau-decay}}
& \leq   (\mathbb{E}[Y_1^{kp}])^{1/p} \ C^{1/q} e^{- \frac{\alpha}{q}  {i_k}} \\
\intertext{and, since $i_k$ is the largest index, $\frac{ {i_1}}{k} + \dots + \frac{ {i_k}}{k} \leq  {i_k}$, hence}
& \leq   (\mathbb{E}[Y_1^{kp}])^{1/p} \ C^{1/q} e^{- \frac{\alpha}{k q} (i_1 + \dots + i_k) }
\end{align*}
which proves the claim.
Now, note that 
\begin{align*}
\Vert w_{\tau(\omega)}(\omega) \Vert^k \leq \left( \Vert g_1 \Vert + \dots + \Vert g_{\tau} \Vert \right)^k & = 
 \left( \sum_{i = 0}^\infty \Vert g_i \Vert \mathbf{1}_{\{ \tau \geq i \}} \right)^k 
 = \sum_{i_1, i_2, \dots, i_k} Y_{i_1} Y_{i_2} \dots Y_{i_k} \cdot \mathbf{1}_{\{ \tau \geq \max \{i_1, \dots,  i_k \} \}}
\end{align*} 
hence by \eqref{E:product}
\begin{align*} 
\mathbb{E}[\Vert w_{\tau(\omega)}(\omega) \Vert^k] 
& \leq  \sum_{i_1, i_2, \dots, i_k} C_3 e^{- \alpha_3  (i_1 + \dots + i_k)} \leq C_3  \left( \sum_{i = 0}^\infty e^{- \alpha_3  {i}} \right)^k < + \infty
\end{align*}
which completes the proof of \autoref{P:finite-moment}.
\end{proof}

We now have all ingredients to prove that the step average of excursions exists and is finite. 

\begin{proof}[Proof of \autoref{P:limit-Xn}]
Let us first note that for all $m,n\in \mathbb{N}$
\begin{align*}
 X_m^{(k)}(\sigma^n \omega)
 & = \sum_{p_H \in [w_n^{-1} p_n, w_n^{-1} p_{n+m}] } E(w_n^{-1} \gamma, H)^k \\
 & = \sum_{w_n p_H \in [p_n, p_{n+m}]} E(\gamma, w_n H)^k \\
 & = \sum_{p_H \in [p_n, p_{n+m}] } E(\gamma, H)^k
 .
\end{align*}
Thus, since every $p_H \in [p_0, p_{n+m}]$ belongs to $[p_0, p_n]$ or to $[p_n, p_{n+m}]$, we obtain
$$X_{n+m}^{(k)}(\omega) \leq X_n^{(k)}(\omega) + X_m^{(k)}(\sigma^n \omega),$$
that is we have subadditivity for $X_n^{(k)}$. By \autoref{P:finite-moment}, $X_n^{(k)}$ is in $L^1(\overline{\mathbb{P}})$, hence the claim follows by Kingman's subadditive ergodic theorem.
\end{proof}

Let us now define a new version of the excursion, where the average is taken over continuous times instead of discrete steps. 
Given an oriented bi-infinite geodesic $\gamma$, let us parameterize it so that $\gamma(0)$ is the closest point projection of the base point $x$ to $\gamma$. 

Define for any time $t \geq 0$,
$$\mathcal{E}^{(k)}(\gamma, t) \coloneqq \sum_{p_H \in [\gamma(0), \gamma(t)]} E(\gamma, H)^k.$$
We will now see that the time average of $\mathcal{E}^{(k)}(\gamma, t)$ is almost surely well-defined. 
For any real $k \geq 1$ and any bi-infinite, oriented geodesic $\gamma$, we define the \emph{average \kth excursion} as 
$$\rho^{(k)}(\gamma) \coloneqq \lim_{t \to +\infty} \frac{\mathcal{E}^{(k)}(\gamma, t)}{t}$$
when it exists. Note that if the limit exists, its value does \emph{not} depend on the choice of $\gamma(0)$.

\begin{prop}[Time average of excursion exists and is finite] \label{cor:existence_rho}
Let $k' > k \geq 1$, and suppose that $\mu$ has finite $(k')^{th}$ moment in some word metric on $\Gamma$ and finite exponential moment in the relative metric. Then for almost every $\omega \in \overline{\Omega}$ the limit
$$\rho^{(k)}(\gamma_\omega) = \lim_{t \to +\infty} \frac{\mathcal{E}^{(k)}(\gamma_\omega, t)}{t}$$
exists and is finite. 
\end{prop}

\begin{proof}
By \autoref{P:limit-Xn}, for almost every $\omega \in \overline{\Omega}$ the limit $\lim_{n \to \infty} \frac{X_n^{(k)}(\omega)}{n}$ exists. Choose such an~$\omega \in \overline{\Omega}$ and let $\gamma = \gamma_\omega$ be the bi-infinite geodesic, parameterized so that $\gamma(0) = p_0$ is the projection of the base point. 
Given $t \geq 0$, there exists a largest $n = n_t$ such that $p_n \in [p_0, \gamma(t)]$. Then by definition $\gamma(t) \in [p_n, p_{n+1}]$.
This implies
$$X^{(k)}_n(\omega) \leq \mathcal{E}^{(k)}(\gamma_\omega, t) \leq X^{(k)}_{n+1}(\omega) $$
so the limit
$$\lim_{t \to +\infty} \frac{\mathcal{E}^{(k)}(\gamma_\omega, t)}{n_t} = \lim_{n \to \infty} \frac{X^{(k)}_n(\omega)}{n}$$
exists.
Moreover, by construction we have $d(p_0, p_n) \leq d(p_0, \gamma(t)) = t \leq d(p_0, p_{n+1})$ and by the triangle inequality, it follows that
$|d(p_0, p_n) - d(x, w_n x)| \leq d(x, p_0) + d(w_n x, p_n)$ (see \autoref{fig:cases_gamma_segment} for illustration). 
Now, by sublinear tracking (\autoref{prop:sublinear_tracking}) 
\begin{equation*}
 \lim_{n \to \infty} \frac{d(w_n x, p_n)}{n} = \lim_{n \to \infty} \frac{d(w_n x, \gamma)}{n} = 0
 .
\end{equation*}
By \autoref{prop:linear_drift} and the remark thereafter, there exists an $L > 0$ such that
$$\lim_{n \to \infty} \frac{d(p_0, p_n)}{n} = L$$
almost surely. Hence we have almost surely that the limit 
$$\lim_{t \to +\infty} \frac{n_t}{t} = \lim_{t \to +\infty} \frac{n_t}{d(p_0, \gamma(t))} = \frac{1}{L}$$
exists and is positive, thus the limit 
$$\rho^{(k)}(\gamma) = \lim_{t \to +\infty} \frac{\mathcal{E}^{(k)}(\gamma, t)}{t} = \lim_{t \to + \infty}  \frac{\mathcal{E}^{(k)}(\gamma, t)}{n_t} \frac{n_t}{t}
= \frac{1}{L} \cdot \lim_{n \to \infty} \frac{X^{(k)}_n(\omega)}{n}$$
also exists. 
\end{proof}

\subsection{Comparison of excursion for one-sided and two-sided random walks} 

We now want to show that the excursion for two-sided random walks is the same as for one-sided random walks. The following proposition establishes this by showing that the excursion does not depend on the backward endpoint. Its proof is based on the fact that two bi-infinite geodesics with the same forward endpoint are exponentially close to each other.

Recall that a bi-infinite geodesic is \emph{recurrent} if its forward endpoint is not the boundary point of a horoball in the collection.

\begin{prop}[Average excursion is independent of backward endpoint] \label{prop:comparison_twosided_with_same_endpoint}
 Let $k, k' > 1$ with $k-1 < k' < k$. Suppose that $\gamma$ and $\gamma'$ are recurrent bi-infinite geodesics with the same forward endpoint, and such that $\limsup\limits_{t \to \infty} \frac{\mathcal{E}^{k'}(\gamma,t) }{t}$ is finite. Then:

 \begin{enumerate}
  \item If $\rho^{(k)}(\gamma)$ exists and is finite, then we have $\rho^{(k)}(\gamma) = \rho^{(k)}(\gamma')$.
  \item If $\rho^{(k)}(\gamma)$ is infinite, then also $\rho^{(k)}(\gamma')$ is infinite.
 \end{enumerate}
\end{prop}

\begin{proof}
In the course of this proof, we will denote as $c_1, c_2, \dots$ constants which depend only on $N$, $k$, and~$\delta$. 

Consider two recurrent bi-infinite geodesics $\gamma$ and $\gamma'$ which have the same forward endpoint. 
There exists a point $p$ on $\gamma$ such that $d(p, \gamma') \leq \delta$. Since the horoballs in the
collection are disjoint with a definite distance $R > 4\delta$, we can choose $p$ such that it has distance at least $2\delta$ to all of the horoballs. Let $p'$ be the closest point projection of $p$ to $\gamma'$.
Since the value of $\rho^{(k)}(\gamma)$ does not depend on the choice of the reference point $\gamma(0)$ along the geodesic, 
we can assume that $\gamma(0) = p$ and $\gamma'(0) = p'$. 

Note that $p'$ has distance at least $\delta$ to all of the horoballs.
Then there exists a sequence of times $t_n \to \infty$ such that for every $n\in \mathbb{N}$, both $\gamma(t_n)$ and $\gamma'(t_n)$ lie in the thick part of $\mathbb{H}^N$.

We now establish an upper bound on $\left| E(\gamma', H)^k - E(\gamma, H)^k \right|$ for all the horoballs $H$, that appear in $\mathcal{E}^{(k)}(\gamma, t_n)$ or $\mathcal{E}^{(k)}(\gamma', t_n)$. 
For this, we consider separately the horoballs $H$ with $E(\gamma', H) \leq E(\gamma, H)$ and the horoballs with
$E(\gamma, H) \leq E(\gamma', H)$.

\begin{figure}
 \centering
 \begin{tikzpicture}
  \draw[fill] (0,0) node[below]{$p$} circle (2pt);
  \draw (0,0) -- (10,0) node[below]{$\gamma$};
  \draw[fill] (0,3) node[above]{$p'$} circle (2pt);
  \draw (0,3) .. controls +(-80:1cm) and +(160:1cm) ..
  (2,0.8) .. controls +(-20:2cm) and +(177:1cm) ..
  (10,0.2) node[above]{$\gamma'$};
  
  \path (1.2,0) node[below]{$q_1$};
  \path (3.8,0) node[below]{$q_2$};
  \path (2.5,1.3) node{$H$};
  \path (2.5,0.23) node{$\widetilde{H}$};
  
  \begin{scope}
   \path[clip] (0,0) -- (10,0) -- (10,3) -- (0,3);
   \draw (2.5,-0.2) circle (1.3cm);
   \draw (2.5,-1.53) circle (2cm);   
  \end{scope}
 \end{tikzpicture}
 \label{fig:rotating_horoball}
 \caption{Setting of \autoref{prop:comparison_twosided_with_same_endpoint}.}
\end{figure}

Let us consider first the case of all the horoballs $H$ with $E(\gamma', H) \leq E(\gamma, H)$. 
Note that $E(\gamma, H)$ is the distance between the entry point $q_1$ and the exit point $q_2$ along the boundary of $H$ (see \autoref{fig:rotating_horoball} for notation); however, the shortest path between $q_1$ and $q_2$ along the boundary of $H$ need not also lie in the geodesic plane $\Pi$ which contains $\gamma$ and~$\gamma'$.
As in the proof of \autoref{lem:excursion_comparison_non-intersecting_geodesics}, we can rotate $H$ around $\gamma$ to obtain another horoball $\widetilde{H}$ which intersects $\gamma$ also in $q_1$ and $q_2$, and whose boundary point is contained in the plane $\Pi$ and is separated from $\gamma'$ by $\gamma$. The shortest path between $q_1$ and $q_2$ along the boundary of $\widetilde{H}$ lies in~$\Pi$, 
and 
$$E(\gamma, H) = E(\gamma, \widetilde{H}).$$

Now, there are two subcases: 
\begin{enumerate}
 \item the rotated horoball $\widetilde{H}$ intersects $\gamma'$;
 \item the rotated horoball $\widetilde{H}$ does not intersect $\gamma'$. 
\end{enumerate}

We concentrate first on the horoballs in the first subcase:
Let $H_1, \dots, H_s$ be the horoballs which appear in $\mathcal{E}^{(k)}(\gamma, t_n)$ (i.e.\ they intersect $[p, \gamma(t_n)]$ and the midpoint of the projection also lies in $[p, \gamma(t_n)]$) with $E(\gamma', H_i) \leq E(\gamma, H_i)$ and such that the rotated horoballs $\widetilde{H}_i$ intersect both~$\gamma$ and~$\gamma'$, numbered in increasing order of distance from $p$. 

Consider now the horoball $H_i$ and denote $d_i \coloneqq d(p, H_i)$ (where we mean the distance along $\gamma$ to the entry point of the horoball).
As $d_i \geq \delta$ for all $i\in \mathbb{N}$, we have by \autoref{L:dist1}
$$0 \leq  E(\gamma, H_i) - E(\gamma', H_i)  \leq c_1 e^{-d_i /2} .$$
Then we have by \autoref{lem:powers_k} with $x=E(\gamma', H_i)$ and $y=E(\gamma, H_i) - E(\gamma', H_i)$
$$0  \leq  E(\gamma, H_i)^k - E(\gamma', H_i)^k \leq s_k c_1 e^{-d_i /2} E(\gamma', H_i)^{k-1} + s_k c_1^k e^{- k d_i/2}$$
where $s_k$ depends only on $k$. Finally, let us note that since the $H_i$ are disjoint, we have 
$$d_{i+1} \geq d_i + \ell_i$$
where $\ell_i$ is the hyperbolic distance between the entry and the exit point of $\gamma$ in $H_i$. Then, applying \autoref{L:comp} to $\widetilde{H}_i$ with $y_1 = e^{d_i}$ we obtain
$$\ell_i \geq 2 \log \left( 1 + \sqrt{2 \delta} e^{- d_i /2} \right) \geq c_2 e^{- d_i /2}$$
where $c_2 = \min \left\{ \inf\limits_{0 < x \leq 1} \frac{2 \log(1 + \sqrt{2 \delta} x)}{x} , 1 \right\} > 0$ only depends on $\delta$, and so 
$$d_{i+1} \geq d_i + c_2 e^{-d_i/2}.$$ 
Applying \autoref{L:recurrence} gives
$$d_i \geq 2 \log \left( 1 + \frac{i c_2}{2} \right),$$
hence
$$e^{-\sfrac{d_i}{2}} \leq \frac{1}{1+\frac{i c_2}{2}} \leq \frac{2}{i c_2}.$$
Together with our earlier considerations, we have 
$$ E(\gamma, H_i) ^k - E(\gamma', H_i)^k \leq \frac{2 s_k c_1}{c_2} \frac{E(\gamma', H_i)^{k-1}}{i} + s_k \left( \frac{2 c_1}{c_2} \right)^k  \frac{1}{i^k}.$$
Now, by H\"older's inequality with $p = \frac{k'}{k-1}$ and $q = \frac{k'}{k'-k+1}$, we obtain
$$ \sum_{i=1}^s E(\gamma', H_i)^{k-1} \cdot \frac{1}{i} \leq \left(  \sum_{i=1}^s E(\gamma', H_i)^{k'} \right)^{\frac{k-1}{k'}}   \left( \sum_{i=1}^s i^{-\frac{k'}{k'-k+1}} \right)^{\frac{k'-k+1}{k'}}$$
and the sum in the second parenthesis is universally bounded in terms of $k$ since $k' > k'-k+1 > 0$. Define $c_ 3 = \frac{2 s_k c_1}{c_2} \left( \sum_{i=1}^s i^{-\frac{k'}{k'-k+1}} \right)^{\frac{k'-k+1}{k'}}$ and $c_4 =  s_k \left( \frac{2 c_1}{c_2} \right)^k   \sum_{i = 1}^\infty \frac{1}{i^k}$. Then we have 
\begin{equation*}
 \sum_{i = 1}^s \left( E(\gamma, H_i)^k - E(\gamma', H_i)^k \right) \leq  c_3 \left(  \sum_{i = 1}^s E(\gamma', H_i)^{k'} \right)^{\frac{k-1}{k'}} + c_4
 .
\end{equation*}

Let us now consider the second subcase, that is when the rotated horoball $\widetilde{H}$ does not intersect~$\gamma'$. 
If we denote as $\mathcal{H}_{small}$ the set of horoballs $H$ for which the rotated horoball $\widetilde{H}$ intersects $\gamma$ and not $\gamma'$, then 
by \autoref{lem_disjoint_horoballs_have_bounded_total_extension} we have 
\begin{equation*} \label{E:type3}
\sum_{H \in \mathcal{H}_{small}} E(\gamma, H)^k  = \sum_{H \in \mathcal{H}_{small}} E(\gamma, \widetilde{H})^k \leq c_5
\end{equation*}
for some constant $c_5 > 0$ that only depends on $\delta$ and $k$.
Summarizing these considerations, we~have
\begin{align*}
 \sum_{\mathclap{\substack{H \in \mathcal{H} \\ E(\gamma', H) \leq E(\gamma, H) }}} \left( E(\gamma, H)^k - E(\gamma', H)^k \right)
 & \leq  c_3 \left(  \sum_{i=1}^s E(\gamma', H_i)^{k'} \right)^{\frac{k-1}{k'}} + c_4 + c_5 \\
 & \leq  c_3 \left(  \sum_{i=1}^s E(\gamma, H_i)^{k'} \right)^{\frac{k-1}{k'}} + c_4 + c_5.
\end{align*}
Now, in the case $E(\gamma, H) \leq E(\gamma', H)$, switching the roles of $\gamma$ and $\gamma'$ yields 
\begin{equation*}
 \sum_{\mathclap{\substack{H \in \mathcal{H} \\ E(\gamma, H) \leq E(\gamma', H) }}} \left( E(\gamma', H)^k - E(\gamma, H)^k \right) \leq  c_3 \left(  \sum_{i = 1}^s E(\gamma, H_i)^{k'} \right)^{\frac{k-1}{k'}} + c_4 + c_5
 .
\end{equation*}
Hence, by putting together both cases, we obtain
$$\left| \mathcal{E}^{(k)}(\gamma, t_n) - \mathcal{E}^{(k)}(\gamma', t_n) \right| \leq 
2c_3 \left( \mathcal{E}^{(k')}(\gamma, t) \right)^{\frac{k-1}{k'}} + 2(c_4+c_5) .$$

Now, recall that $\limsup_{t \to \infty} \frac{\mathcal{E}^{(k')} (\gamma, t)}{t}$ exists by assumption.
This implies that 
$$\lim_{t \to \infty} \frac{  \left( \mathcal{E}^{(k')} (\gamma, t) \right)^{\frac{k-1}{k'}} }{t}
= \limsup_{t \to \infty} \frac{ \left( \frac{ \mathcal{E}^{(k')}(\gamma, t) }{t} \right)^{\frac{k-1}{k'}} }{t^{\frac{k'-k+1}{k'}} } = 0$$
hence 
$$\rho^{(k)}(\gamma') = \lim_{n \to \infty} \frac{\mathcal{E}^{(k)} (\gamma', t_n)}{t_n} = \lim_{n \to \infty} \frac{\mathcal{E}^{(k)} (\gamma, t_n)}{t_n} = \rho^{(k)}(\gamma)$$
which shows both of the claimed statements.
\end{proof}

\begin{cor}[Average excursion exists and is finite] \label{cor:excursion_random_walks_linear}
 Let $k' > k  > 1$, and suppose that $\mu$ has finite $(k')^{th}$ moment in some word metric on $\Gamma$
 and finite exponential moment in the relative metric. Let $\nu$ be the hitting measure. Then for $\nu$--almost every~$\xi \in \partial \mathbb{H}^N$ and $\gamma$ the geodesic ray from $x$ to $\xi$, the average \kth excursion
 \begin{equation*}
  \rho^{(k)}(\gamma) = \lim_{t \to \infty} \frac{\mathcal{E}^{(k)}(\gamma, t)}{t}
 \end{equation*}
 exists and is finite.
 
\begin{proof}
 By definition of the hitting measure and because it is not atomic, for $\nu$--almost every~$\xi \in \partial \mathbb{H}^N$, there exists an $\omega \in \overline{\Omega}$ such that $\gamma_\omega$ is recurrent and its forward endpoint is $\xi$.
 Choose $k' > 1$ such that $k-1 < k' < k$.
 By \autoref{cor:existence_rho} for $\overline{\mathbb{P}}$--almost every such $\omega$, we have that $\rho^{(k)}(\gamma_\omega)$ and~$\rho^{(k')}(\gamma_\omega)$ exist.
 
 Now let $\gamma'$ be the bi-infinite geodesic through the base point $x$ and with forward endpoint $\xi$. Then we have by \autoref{prop:comparison_twosided_with_same_endpoint} that $\rho^{(k)}(\gamma) = \rho^{(k)}(\gamma') = \rho^{(k)}(\gamma_\omega)$ which exists.
\end{proof}
\end{cor}

This corollary completes the proof of the first part of \autoref{thm:cusp_excursion}.

\section{Excursion for the Lebesgue measure} \label{sec:excursion_Lebesgue_case}

Recall that $\Gamma$ is a discrete subgroup of $\Isom(\mathbb{H}^N)$ such that $M = \mathbb{H}^N/\Gamma$ is a hyperbolic orbifold of finite volume with cusps. Let $T^1M$ be its unit tangent bundle and for every $v \in T^1 M$, let $p_v$ be its projection to $M$.
Furthermore, let us denote by $\widetilde{\lambda}$ the normalized Liouville measure on $T^1 M$, where every fibre has measure $1$. The normalized Liouville measure is invariant under the geodesic flow by construction.
Recall that we have chosen a $\Gamma$--invariant collection $\mathcal{H}$ of disjoint horoballs in $\mathbb{H}^N$, and let $\Mthick \coloneqq (\mathbb{H}^N \setminus \mathcal{H})/\Gamma$ be the thick part of the quotient as before.

Define for any $k \geq 1$ the function $f_k : T^1 M \to \mathbb{R}$ as 
$$f_k(v) \coloneqq e^{k \cdot d(p_v, \Mthick)},$$ 
where $d(p_v, \Mthick)$ is the distance between the 
projection of the vector $v \in T^1 M$ and the thick part.

Let $\phi_t \colon T^1 M \to T^1 M$ be the geodesic flow which is ergodic. Then by the ergodic theorem, for any measurable non-negative $f \colon T^1 M \to \mathbb{R}$ and almost every starting vector $v$,
\begin{equation} \label{E:ergodic-flow}
 \lim_{t \to \infty} \frac{1}{t} \int_0^t f(\phi_s v) \ ds = \int_{T^1 M} f \ d\widetilde{\lambda}.
\end{equation}

\begin{lem} \label{lem:integral_f_k}
The integral 
$$\int_{T^1 M} f_k \ d \widetilde{\lambda}$$ 
is finite for $1 \leq k < N -1$, and infinite for $k \geq N-1$.
\end{lem}

\begin{proof}
We have
\begin{equation*}
 \int_{T^1 M} f_k \ d \widetilde{\lambda} = \int_{T^1 \Mthick} 1 \ d \widetilde{\lambda} + \int_{T^1 (M\setminus \Mthick)} f_k \ d \widetilde{\lambda}
 .
\end{equation*}
Note that the first integral on the right hand side is finite, hence we are only interested in the second integral. 

Every cusp in the quotient $M$ has a neighborhood which can be lifted to hyperbolic $N$--space~$\mathbb{H}^N$ to form a subset 
$S_\Delta \coloneqq \{ (x_1, x_2, \dots, x_N) \ : \ (x_2, \dots, x_{N}) \in \Delta, x_1 > 1 \} \subseteq \mathbb{H}^N$, where $\Delta$ is a compact set in $\mathbb{R}^{N-1}$. The set $S_\Delta$ can be chosen such that the projection of the interior of $S_\Delta$ to~$M$ is injective and a local isometry.
The hyperbolic volume form is 
$$dV = \frac{dx_1 \ dx_2 \ \ldots \ dx_N}{(x_1)^N}$$
and with $d(p_v, \Mthick) = \log (x_1)$ for every $p_v \notin \Mthick$ with first coordinate $x_1$, we have
$$\int_{T^1 (M\setminus \Mthick)} f_k \ d \widetilde{\lambda} = \int_{S_\Delta} (x_1)^k \ dV = \int_{S_\Delta} (x_1)^k \frac{dx_1 \ dx_2 \ \ldots \ dx_N}{(x_1)^N}
= \int_{\Delta} dx_2 \ \ldots \ dx_N \int_1^{+\infty}  \frac{dx_1 \ }{(x_1)^{N-k}}$$
so the integral diverges if $k \geq N-1$ and converges for $k < N-1$.
\end{proof}

\begin{lem} \label{lem:comparison_excursion_f_k}
Let $\gamma$ be a geodesic ray in $\mathbb{H}^N$ with starting direction $v$ and $H \in \mathcal{H}$ be a horoball, that intersects $\gamma$ with entry point $\gamma(t_1)$ and exit point $\gamma(t_2)$.
For every $k\geq 1$, there exists a constant~$c > 0$ which depends only on $k$ such that
$$\frac{1}{c} \int_{t_1}^{t_2} f_k(\phi_t v)\, dt \leq E(\gamma, H)^k \leq c \int_{t_1}^{t_2} f_k(\phi_t v) \, dt.$$
\end{lem}

\begin{proof}
As in the previous lemma, we can choose coordinates $(x_1, x_2, \dots, x_N)$ 
such that the metric outside of the thick part is given by 
$$ds^2 = \frac{dx_1^2 + dx_2^2 + \dots + dx_N^2}{x_1^2}.$$
and such that $\gamma$ lies in the $(x_1, x_2)$--plane.

\begin{figure}
 \begin{center}
  \begin{tikzpicture}[scale=2]
   \draw[->, gray] (-1,0) -- (4,0);
   \draw[->, gray] (0,0) node[below] {$0$} -- (0,2.3);
   
   \draw (-1,1) -- (4,1) node[right] {$H$};
   
   \fill (2,0) circle (1pt);
   \draw (3.7,0) arc (0:180:1.7);
   \draw[<->] (0.65,0.95) -- (3.35,0.95);
   \path (2,1) node[below] {$E(\gamma, H)$};
   
   \path (0.7,1) node[above left] {$\gamma(t_1)$};
   \path (3.3,1) node[above right] {$\gamma(t_2)$};
   
   \draw[<->] (0.35,-0.05) -- (1.95,-0.05);
   \path (1.15,0) node[below] {$R$};
   
   \draw (2,0) -- ++(144:1.7);
   \draw (1.6,0) arc (180:144:0.4);
   \path (1.75,0.07) node{$\theta_1$};
  \end{tikzpicture}
  \caption{Setting of \autoref{lem:comparison_excursion_f_k}.}
  \label{fig:calculation_Lebesgue}
 \end{center}
\end{figure}

Let $x_1(t)$ be the first coordinate and $x_2(t)$ be the second coordinate of $\gamma(t)$.
Now for $t_1 \leq t \leq t_2$, the distance from $\gamma(t)$ to the thick part is $\int_1^{x_1(t)} \frac{ds}{s} = \log x_1(t)$.
Then
$$\int_{t_1}^{t_2} f_k(\phi_t v) \ dt = \int _{t_1}^{t_2} (x_1(t))^k \ dt .$$
Now, we can parameterize the geodesic $\gamma(t)$ by $x_1(t) = R\sin \theta(t)$ and $x_2(t) = R\cos \theta(t)$ for some~$R > 1$ (see \autoref{fig:calculation_Lebesgue}).
As $\gamma$ has unit speed in the hyperbolic metric, we have
\begin{equation*}
 1 = \left\| \phi_t v \right\| = \frac{\sqrt{x_1'(t)^2 + x_2'(t)^2}}{x_1(t)} = \frac{\theta'(t)}{\sin \theta(t)}
 ,
\end{equation*}
hence 
$d t = \frac{d\theta}{\sin \theta}$.
Then
$$\int_{t_1}^{t_2} (x_1(t))^k \ dt = \int_{\theta_1}^{\theta_2} R^k (\sin \theta)^k \frac{d\theta}{\sin \theta} = $$
and by using the substitution $u = \sin \theta$ and $R \sin \theta_1 = x_1(t_1) = 1 = x_2(t_2) = R \sin \theta_2$
$$ = 2 \int_{1/R}^1 R^k u^{k-1} \frac{du}{\sqrt{1 - u^2}} = c_k R^k + O(1)$$
where $c_k = 2 \int_{0}^1 \frac{ u^{k-1} \, du}{\sqrt{1 - u^2}}$.
The excursion, on the other hand, is 
$$E(\gamma, H) = d_{\partial H}(\gamma(t_1), \gamma(t_2)) = 2 R \cos \theta_1 = 2 R \sqrt{1 - R^{-2}}.$$

Comparing the last two equations shows that we can choose a constant $c > 0$ as desired.
\end{proof}

With the previous two lemmas, we can now show that the \nminusoneth excursion $\mathcal{E}^{(N-1)}(\gamma,t)$ grows superlinearly in $t$ for rays $\gamma$, that are generic with respect to Lebesgue measure.

\begin{prop}[Average excursion is infinite] \label{cor:excursion_Lebesgue_case_superlinear}
Let $x\in \mathbb{H}^N$ be a base point.
For $\lambda$--almost every $\xi \in \partial \mathbb{H}^N$ and $\gamma$ the corresponding geodesic ray from $x$ to $\xi$, the average \nminusoneth excursion
\begin{equation*}
 \rho^{(N-1)}(\gamma) = \lim_{t \to +\infty} \frac{\mathcal{E}^{(N-1)}(\gamma, t)}{t}
\end{equation*}
is infinite.

\begin{proof}
Since the geodesic flow on $T^1 M$ is ergodic, 
almost every geodesic ray spends a linear amount of time in any $\epsilon$--neighborhood of the boundary of the thick part. 
Thus, for almost every geodesic ray $\gamma$, there exists a constant $c_1 >0$ and an infinite sequence $(t_n)$ of times where $\gamma$ enters a horoball and such that 
$\frac{t_n}{n} \to c_1$, thus $\frac{t_n}{t_{n+1}} \to 1$.
Note that in particular $\gamma(t_n)$ lies on the boundary of the thick part.

Now, by summing up over all horoballs which $\gamma$ enters up to time $t_n$ and applying \autoref{lem:comparison_excursion_f_k}, we have
\begin{equation*}
 c \cdot \mathcal{E}^{(N-1)}(\gamma, t_n) \geq \int_{0}^{t_n} f_k(\phi_s v) \, ds - t_n.
\end{equation*}
By the ergodic theorem, for almost every starting vector $v$, one has
\begin{equation*}
 \lim_{t \to \infty} \frac{1}{t} \int_0^t f_k(\phi_s v) \ ds= \int_{T^1 M} f_k \ d\widetilde{\lambda}.
\end{equation*}
We have shown in \autoref{lem:integral_f_k} that the right hand side is infinite for $k\geq N-1$.
Now, for every~$t \geq 0$ there exist $t_n$ and $t_{n+1}$ with $t_n \leq t \leq t_{n+1}$. This implies
\begin{equation*}
 \frac{\mathcal{E}^{(N-1)}(\gamma, t)}{t}
 \geq\frac{\mathcal{E}^{(N-1)}(\gamma, t_n)}{t_{n+1} }
 = \frac{\mathcal{E}^{(N-1)}(\gamma, t_n)}{t_n} \cdot \frac{t_n}{t_{n+1}}
 \geq \frac{\int_0^{t_n} f_{N-1}(\phi_s v) \ ds - t_n}{c \cdot t_n} \cdot \frac{t_n}{t_{n+1}} .
\end{equation*}
Then we obtain
\begin{align*}
 \rho^{(N-1)}(\gamma)
 & = \lim\limits_{t \to \infty} \frac{\mathcal{E}^{(N-1)}(\gamma, t)}{t}
 \geq \lim\limits_{n \to \infty} \frac{\int_0^{t_n} f_{N-1}(\phi_s v) \ ds - t_n}{c \cdot t_n} \cdot \frac{t_n}{t_{n+1}}
 = \frac{ \int_{T^1 M} f_{N-1} \ d \widetilde{\lambda} - 1}{c} = + \infty
 .
\end{align*}

With a similar argument as above using the upper bound from \autoref{lem:comparison_excursion_f_k} and that $\int_{T^1 M} f_k \ d\widetilde{\lambda}$ is finite for $k< N-1$ from \autoref{lem:integral_f_k}, we obtain for almost every starting vector $v$ of a geodesic ray $\gamma$ that
\begin{equation*}
 \limsup_{t \to \infty} \frac{\mathcal{E}^{(k)}(\gamma, t)}{t}
\end{equation*}
is finite for every $k<N-1$.

Combining the two arguments, we obtain a subset of $T^1 \mathbb{H}^N$ of full measure such that for the corresponding geodesic rays $\gamma$ the $\limsup_{t \to \infty} \frac{\mathcal{E}^{(N - \sfrac{3}{2})}(\gamma, t)}{t}$ is finite and $\rho^{(N-1)}(\gamma)$ is infinite.
This implies that for $\lambda$--almost every $\xi \in \partial \mathbb{H}^N$, there exists a geodesic $\gamma'$ with forward boundary point $\xi$ and a starting vector $v \in T^1 \mathbb{H}^N$ from this full measure set.
Let now $\gamma$ be the geodesic ray from the base point $x$ to the forward boundary point of $\gamma'$.
Then by \autoref{prop:comparison_twosided_with_same_endpoint}, the average \nminusoneth excursion $\rho^{(N-1)}(\gamma)$ is also infinite.
This shows that for $\lambda$--almost every $\xi \in \partial \mathbb{H}^N$, the average \nminusoneth excursion of the geodesic from $x$ to $\xi$ is infinite.
\end{proof}
\end{prop}

This proposition completes the proof of the second part of \autoref{thm:cusp_excursion}.
Note that the previous proof also shows that $N-1$ is the smallest $k$ for which the \kth average excursion is infinite as $\limsup_{t \to \infty} \frac{\mathcal{E}^{(k)}(\gamma, t)}{t}$ is finite almost surely for every $k<N-1$.

\bibliographystyle{amsalpha}
\bibliography{literature_excursion_Hn}

\end{document}